\documentclass[11pt,letterpaper]{article}
\usepackage[T1]{fontenc}
\usepackage[utf8]{inputenc}
\usepackage[letterpaper,margin=1in]{geometry}
\usepackage{newtxtext}
\usepackage{amsmath,amssymb,amsthm}
\usepackage{newtxmath}
\allowdisplaybreaks[1]
\usepackage{bm}
\usepackage{microtype}
\usepackage{enumitem}
\usepackage{tabularx,multirow,booktabs}
\usepackage[group-separator={,},group-minimum-digits=4]{siunitx}
\usepackage{graphicx}
\usepackage{xcolor}
\usepackage[font=small,labelfont=bf]{caption}
\usepackage[authoryear,round]{natbib}
\bibpunct[, ]{(}{)}{,}{a}{}{,}
\usepackage[hidelinks,unicode]{hyperref}
\hypersetup{
  pdftitle={An Exact Branch-Price-and-Cut Algorithm for a Truck-as-Hub Vehicle Routing Problem with Drones},
  pdfauthor={Fan Pu; Sivakumar Rathinam; Yunlong Zhang; Xiaopeng Li; Xiubin Wang; Yang Zhou},
  pdfsubject={Rural last-mile delivery and exact optimisation},
  pdfkeywords={rural last-mile delivery, vehicle routing problem, drone logistics, branch-price-and-cut, parallel computing}
}

\newtheorem{theorem}{Theorem}
\newtheorem{proposition}{Proposition}
\setlist[itemize]{topsep=4pt,itemsep=2pt,parsep=0pt}

\newcommand{\nodes}[1]{\boldsymbol{\delta}^{#1}}
\newcommand{\Graph}{\mathbf{G}}
\newcommand{\TransGraph}{\mathbf{G}^\top}
\newcommand{\Nodes}{\mathbf{N}} 
\newcommand{\Arcs}{\mathbf{A}}
\newcommand{\CustNodes}{\Nodes^\textnormal{C}}
\newcommand{\HubNodes}{\Nodes^\textnormal{H}}
\newcommand{\DemandWeight}[1]{w_{#1}}
\newcommand{\TruckofDrone}{k_d}
\newcommand{\Trucks}{\mathbf{K}}
\newcommand{\Drones}{\mathbf{D}}
\newcommand{\TruckArcs}{\Arcs^\textnormal{truck}}
\newcommand{\DroneArcs}{\Arcs^\textnormal{drone}}
\newcommand{\DroneFleetSize}{\Theta}
\newcommand{\TruckCost}{c^\textnormal{truck}}
\newcommand{\DroneCost}{c^\textnormal{drone}}

\newcommand{\ObjDelay}{f^{\textnormal{delay}}}
\newcommand{\ObjReturn}{f^{\textnormal{return}}}

\newcommand{\WeightDelay}{\eta^{\textnormal{delay}}}
\newcommand{\WeightReturn}{\eta^{\textnormal{return}}}
\newcommand{\WeightCost}{\eta^{\textnormal{cost}}}
\newcommand{\ObjCoeffDelay}{\rho^{\textnormal{delay}}}
\newcommand{\ObjCoeffReturn}{\rho^{\textnormal{return}}}
\newcommand{\ObjCoeffCost}{\rho^{\textnormal{cost}}}

\newcommand{\ObjLB}[1]{\underline{f}^{#1}}
\newcommand{\ObjUB}[1]{\overline{f}^{#1}}
\newcommand{\delayUB}{\ObjUB{\textnormal{delay}}}
\newcommand{\delayLB}{\ObjLB{\textnormal{delay}}}
\newcommand{\returnUB}{\ObjUB{\textnormal{return}}}
\newcommand{\returnLB}{\ObjLB{\textnormal{return}}}
\newcommand{\costUB}{\ObjUB{\textnormal{cost}}}
\newcommand{\costLB}{\ObjLB{\textnormal{cost}}}

\newcommand{\MaxTruckTravelTime}{\overline{t}^{\textnormal{truck}}}
\newcommand{\MaxDroneTravelTime}{\overline{t}^{\textnormal{drone}}}
\newcommand{\MaxDroneTripTime}{\overline{t}^{\textnormal{trip}}}

\newcommand{\MinTruckNum}{\underline{K}}
\newcommand{\MaxUsedTruckNum}{\overline{K}}
\newcommand{\MinDroneFlightNum}{\underline{D}}
\newcommand{\TruckArrTimeLB}[1]{\underline{a}_{#1}^\textnormal{truck}}
\newcommand{\ShortestTruckTime}[2]{\ell^{\textnormal{truck}}_{#1,#2}}
\newcommand{\RouteTruckArcLimit}{L^{\textnormal{truck}}}
\newcommand{\SortedTruckArcTime}[1]{\overline{t}^{\textnormal{truck}}_{[#1]}}
\newcommand{\HubWaitUB}[1]{\overline{W}_{#1}}
\newcommand{\RouteTimeUB}{\overline{T}^{\textnormal{route}}}
\newcommand{\ServTimeUB}[1]{\overline{a}_{#1}}
\newcommand{\ReqServUB}[1]{\widehat{a}^{\textnormal{req}}_{#1}}
\newcommand{\ServWindowOffset}[1]{\Delta_{#1}^{\textnormal{win}}}
\newcommand{\EffServUB}[1]{\widetilde{a}^{\max}_{#1}}
\newcommand{\LabelServLB}[2]{\underline{a}_{#2}(#1)}
\newcommand{\DeadlineReachSet}[1]{\mathcal{U}^{\textnormal{time}}(#1)}
\newcommand{\TruckServUB}[1]{\overline{a}_{#1}^{\textnormal{truck}}}
\newcommand{\DroneServUB}[1]{\overline{a}_{#1}^{\textnormal{drone}}}
\newcommand{\TruckServLB}[1]{\underline{a}_{#1}^{\textnormal{truck}}}
\newcommand{\DroneServLB}[1]{\underline{a}_{#1}^{\textnormal{drone}}}
\newcommand{\DroneFeasibleHubs}[1]{\mathcal{H}^{\textnormal{D}}_{#1}}
\newcommand{\MandatoryDroneCust}{\mathcal{C}^{\textnormal{D-only}}}
\newcommand{\PositivePart}[1]{\left[#1\right]_+}
\newcommand{\CompletionReturnLB}[1]{\underline{T}^{\textnormal{return}}(#1)}
\newcommand{\ResidualPayload}[1]{Q^{\textnormal{res}}(#1)}
\newcommand{\ReachableCustSet}[1]{\mathcal{U}(#1)}
\newcommand{\DualRewardBound}[1]{\overline{R}(#1)}
\newcommand{\FarkasDualRewardBound}[1]{\overline{R}^{\textnormal{F}}(#1)}
\newcommand{\FarkasReducedCostLB}[1]{\underline{c}^{\textnormal{F,comp}}(#1)}
\newcommand{\KnapReducedCostLB}[1]{\underline{c}^{\textnormal{comp}}(#1)}
\newcommand{\RouteServedSet}[1]{\mathcal{C}(#1)}
\newcommand{\BatchLimit}{B^{\max}}
\newcommand{\FarkasBatchLimit}{B^{\textnormal{F},\max}}
\newcommand{\SRCutBatchLimit}{B^{\textnormal{SR}}}

\newcommand{\CompletionSet}[1]{\mathcal{E}^{\textnormal{complete}}_{#1}}
\newcommand{\Completion}{\boldsymbol{\upsilon}}
\newcommand{\CompletionCost}[2]{\overline c(#1\oplus #2)}
\newcommand{\KnapVar}[1]{\xi_{#1}}
\newcommand{\PhysLoc}[1]{\operatorname{truck}(#1)}
\newcommand{\BlockPos}[1]{\operatorname{cust}(#1)}
\newcommand{\RetCredit}[2]{\Delta^{\textnormal{return}}(#1,#2)}
\newcommand{\BranchState}[1]{\mathbf{b}_{#1}}
\newcommand{\BranchExtSet}[1]{\mathcal{E}^{\textnormal{remain}}_{#1}}
\newcommand{\VisitedTruckNodes}[1]{\mathcal{V}_{#1}}

\newcommand{\RoutePool}{\mathcal{P}}
\newcommand{\NodeRoutePool}[1]{\mathcal{P}_{#1}}
\newcommand{\BranchFeasibleRoutes}[1]{\mathbf{R}_{#1}^{\textnormal{adm}}}
\newcommand{\CandidateSRset}[1]{\mathbf{\Pi}_{#1}^{\textnormal{all}}}

\newcommand{\arriveLB}[1]{\underline{a}_{#1}}
\newcommand{\DroneEndurance}{e^\textnormal{drone}}
\newcommand{\DronePayload}{p^\textnormal{drone}}
\newcommand{\TruckPayload}{p^\textnormal{truck}}
\newcommand{\TruckTravelTime}[2]{t_{#1,#2}^\textnormal{truck}}
\newcommand{\DroneTravelTime}[2]{t_{#1,#2}^\textnormal{drone}}
\newcommand{\DroneTripTime}[2]{\tau_{#1,#2}^\textnormal{drone}}
\newcommand{\PayloadLoad}[2]{w^{\textnormal{pay}}_{#1,#2}}
\newcommand{\OrderLoad}[2]{o_{#1,#2}}
\newcommand{\BigMLoad}{M^{\textnormal{load}}}
\newcommand{\BigMOrder}{M^{\textnormal{order}}}
\newcommand{\BigMTime}{M^{\textnormal{time}}}
\newcommand{\RouteObjShift}{\Delta^{\textnormal{const}}}
\newcommand{\PadActivated}[2]{u_{#1,#2}}
\newcommand{\DroneBlock}[2]{\mathbf{B}_{#2,#1}}
\newcommand{\DroneOutArcs}[1]{\nodes{\textnormal{drone}+}_{#1}}
\newcommand{\TruckOutArcs}[1]{\nodes{\textnormal{truck}+}_{#1}}
\newcommand{\TruckInArcs}[1]{\nodes{\textnormal{truck}-}_{#1}}
\newcommand{\DroneInArcs}[1]{\nodes{\textnormal{drone}-}_{#1}}
\newcommand{\DepotSource}{\underline{n}}
\newcommand{\DepotSink}{\overline{n}}
\newcommand{\ObjCost}{f^{\textnormal{cost}}}
\newcommand{\DronesOfTruck}[1]{\Drones_{#1}}
\newcommand{\RouteSet}{\mathbf{R}}
\newcommand{\RouteSubset}[1]{\mathbf{R}_{#1}}
\newcommand{\RouteCost}[1]{c_{#1}}
\newcommand{\TruckRoute}[1]{{#1}^\textnormal{truck}}
\newcommand{\AbbrProb}{\textnormal{TH-VRPD}}
\newcommand{\SRset}{\mathbf{\Pi}}
\newcommand{\ActiveSRset}[1]{\mathbf{\Pi}_{#1}}
\newcommand{\SRsubset}{\mathbf{\pi}}
\newcommand{\DroneAccessNodes}[1]{\Nodes^\textnormal{drone}_{#1}}
\newcommand{\TruckAccessNodes}[1]{\Nodes^\textnormal{truck}_{#1}}

\newcommand{\Duplicate}[1]{#1^\top}
\newcommand{\Origin}[1]{\operatorname{orig}(#1)}
\newcommand{\ServedCust}[1]{\operatorname{cust}(#1)}

\newcommand{\CustNodesDuplicate}{\Nodes^{\textnormal{C}^\top}}
\newcommand{\CustHubArcsDuplicate}{\Arcs^{\textnormal{HC}^\top}}
\newcommand{\CustDupCustDupArcs}{\Arcs^{\textnormal{C}^\top \textnormal{C}^\top}}
\newcommand{\CustDupCustArcs}{\Arcs^{\textnormal{C}^\top \textnormal{C}}}
\newcommand{\TravelTime}[2]{t_{#1, #2}}

\newcommand{\RouteTravelArc}[3]{\Phi_{#1, #2, #3}}
\newcommand{\SameRouteFlow}[2]{\chi_{#1,#2}}
\newcommand{\TransArcFlow}[3]{\gamma^\top_{#1,#2,#3}}

\newcommand{\BranchArcDown}[3]{\mathcal{R}^{\top,0}_{#1,#2,#3}}
\newcommand{\BranchArcUp}[3]{\mathcal{R}^{\top,1}_{#1,#2,#3}}

\newcommand{\BranchTogether}[1]{\mathcal{R}^{=}_{#1}}
\newcommand{\BranchSeparate}[1]{\mathcal{R}^{\neq}_{#1}}

\newcommand{\RouteArrTime}[2]{a_{#1,#2}}
\newcommand{\ForwardLabel}[1]{\Lambda_{#1}}
\newcommand{\PartialPath}[1]{\Nodes_{#1}}
\newcommand{\TruckWeight}[1]{w_{\textnormal{truck},#1}}
\newcommand{\LabelHub}[1]{h_{#1}}
\newcommand{\TruckArrHubTime}[1]{a_{\textnormal{pad},#1}}
\newcommand{\TruckWaitHubTime}[1]{t_{\textnormal{wait},#1}}
\newcommand{\HubFlyDroneNum}[1]{z_{\textnormal{pad},#1}}
\newcommand{\NumCustPiCoverByLabel}[2]{\psi_{{#1}, {#2}}}
\newcommand{\NumCustPiCoverByLabelSet}[1]{\mathbf{\Psi}_{#1}}
\newcommand{\LabelReducedCost}[1]{\overline{c}_{#1}}
\newcommand{\ForwardLabelInfo}[2]{
\left(
\PartialPath{#1}^{#2},~
\mathcal{V}_{#1}^{\textnormal{T},#2},~
\RealizedServTime{#1}^{#2},~
\TruckWeight{#1}^{#2},~
\LabelHub{#1}^{#2},~
\TruckArrHubTime{#1}^{#2},~
\TruckWaitHubTime{#1}^{#2},~
\HubFlyDroneNum{#1}^{#2},~
\NumCustPiCoverByLabelSet{#1}^{#2}
=
\{
\NumCustPiCoverByLabel{#1}{\SRsubset}^{#2}:
\SRsubset^{#2}\in\SRset
\},~
\mathbf{b}_{#1}^{#2},~
\LabelReducedCost{#1}^{#2}
\right)
}

\newcommand{\SourceNeighborSet}[1]{\mathcal{N}^{\textnormal{first}}_{#1}}
\newcommand{\SourceNeighborBlock}[2]{\mathcal{N}^{\textnormal{first}}_{#1,#2}}
\newcommand{\PricingWorkerNum}[1]{K^{\textnormal{par}}_{#1}}
\newcommand{\WorkerRouteSet}[2]{\mathcal{R}^{\textnormal{adm}}_{#1,#2}}
\newcommand{\SourceWorkScore}[2]{\widehat W_{#1}(#2)}
\newcommand{\OpenTaskLabels}[1]{\mathcal L^{\textnormal{open}}_{#1}}
\newcommand{\TaskClosureGap}[1]{G_{#1}}
\newcommand{\TaskWorkEstimate}[1]{\widehat R_{#1}}

\newcommand{\TruckUseArc}[3]{x_{#1,#2,#3}}
\newcommand{\DroneUseArc}[3]{y_{#1,#2,#3}}
\newcommand{\RealizedServTime}[1]{a_{#1}}
\newcommand{\WaitTime}[2]{t_{#1,#2}^\textnormal{wait}}
\newcommand{\TruckArrTime}[2]{a_{#1,#2}^\textnormal{truck}}
\newcommand{\DroneArrTime}[2]{a_{#1,#2}^\textnormal{drone}}
\newcommand{\RouteVisitNode}[2]{\theta_{#1,#2}}
\newcommand{\NumCust}{C}
\newcommand{\NumHub}{H}
\newcommand{\DroneAccessNum}[1]{q_{#1}}
\newcommand{\TruckAccessNum}[1]{s_{#1}}
\newcommand{\RouteSelect}[1]{z_{#1}}

\newcommand{\SRCorrectionSet}[2]
{\widehat{\SRset}\!\left(#1,#2\right)}

\makeatletter
\newcommand{\compactappendixparagraphs}{%
  \renewcommand\paragraph{\@startsection{paragraph}{4}{\z@}%
    {1.5ex \@plus .3ex \@minus .2ex}{-1em}%
    {\normalfont\normalsize\bfseries}}}
\makeatother

\title{\LARGE\bfseries An Exact Branch-Price-and-Cut Algorithm for a Truck-as-Hub Vehicle Routing Problem with Drones}
\author{%
\normalsize Fan Pu$^{1}$, Sivakumar Rathinam$^{2}$, Yunlong Zhang$^{1}$,\\
\normalsize Xiaopeng Li$^{3}$, Xiubin Wang$^{1}$, Yang Zhou$^{1}$\\[5pt]
\footnotesize $^{1}$Department of Civil and Environmental Engineering, Texas A\&M University, TX 77843\\
\footnotesize $^{2}$Department of Mechanical Engineering, Texas A\&M University, TX 77843\\
\footnotesize $^{3}$Department of Civil and Environmental Engineering, University of Wisconsin-Madison, WI 53706\\[4pt]
\footnotesize\texttt{fan-pu@tamu.edu; srathinam@tamu.edu; yzhang@civil.tamu.edu}\\
\footnotesize\texttt{xli2485@wisc.edu; bwang@tamu.edu; yangzhou295@tamu.edu}%
}
\date{}

\begin{document}
\maketitle
\begin{abstract}
Rural last-mile logistics remains challenging because of sparse demand, limited road infrastructure, and long travel distances. We study a Truck-as-Hub Vehicle Routing Problem with Drones in Rural Areas, in which multiple trucks each carry a dedicated fleet of drones. Drone launch and retrieval are restricted to predesignated passive synchronization pads, where a visiting truck temporarily acts as a mobile hub, may dispatch multiple drones simultaneously, and must wait until all launched drones return before continuing its route. We formulate a mixed-integer quadratic programming model that jointly accounts for customer delivery delays, truck return times, and operating costs. To solve the problem exactly, we transform the original truck--drone network so that each feasible route can be represented as a single path and develop a branch-price-and-cut algorithm with tailored pricing, cutting, branching, and parallel-computing procedures. Computational experiments show that the proposed algorithm proves optimality for most large-scale test cases and achieves optimality gaps of no more than 4\% for the remaining cases, whereas Gurobi reaches the 1,800-second time limit on all large-scale instances when applied directly to the compact formulation. Sensitivity analyses show that drone operating costs and managerial priorities substantially affect drone utilization and customer delays, providing practical guidance for balancing service quality and operating cost in rural delivery systems.
\end{abstract}

\section{Introduction}\label{intro}
Rural last-mile delivery is particularly challenging due to sparse demand, limited road accessibility, and long travel distances. Trucks offer substantial range and payload capacity, while drones can extend service coverage to hard-to-reach locations but are constrained by limited payload and battery endurance. Hybrid truck--drone systems leverage these complementary strengths: trucks provide long-haul transportation and serve as mobile drone carriers, while drones handle selected short-range deliveries \citep{murray2015flying,agatz2018optimization}.

One key challenge in truck–drone delivery lies in coordinating the two vehicle types. In many formulations, a drone departs from its companion truck, flies to a customer to deposit a package, and then returns to the truck for battery replacement or recharging before its next flight. Meanwhile, the truck continues ground deliveries, proceeding to the next scheduled rendezvous point where the drone will rejoin it \citep{ha2018min,carlsson2018coordinated}.
This hybrid system has been formalized in two related problems: the \emph{Flying Sidekick Traveling Salesman Problem} (FSTSP) \citep{murray2015flying}, which models drone missions as sorties, and the \emph{Traveling Salesman Problem with Drone} (TSP-D) \citep{agatz2018optimization}, which treats each drone trip as an operation.
In the FSTSP, a drone may launch from and rejoin the truck at different locations \citep{murray2015flying}, with subsequent work considering alternative objectives such as operating cost \citep{ha2018min}. In the TSP-D, drone trips are represented as operations associated with the truck route \citep{agatz2018optimization}.

The FSTSP and TSP-D literature shows that simultaneous drone and truck operation boosts routing efficiency and cuts idle time. In practice, however, traffic variability and unpredictable road conditions complicate precise rendezvous timing, introducing significant uncertainty. To keep the problem tractable, most studies therefore focus on a single truck–single drone pairing, sidestepping the added complexity of larger fleets.
A natural extension considers multiple drones carried by one truck. Some approaches require the truck to wait at a launch location until all dispatched drones return \citep{moshref2017using}, whereas others extend TSP- or FSTSP-based formulations to multiple drones \citep{cavani2021exact,murray2020multiple}.

With multiple trucks and drones, solving the resulting problem to optimality becomes substantially more challenging. Recent exact methods for related truck--drone variants exploit route-space formulations within branch-price-and-cut or branch-and-price frameworks \citep{xia2023branch,li2023branch,yin2023branch,joo2025branch,schmidt2025exact}. Depending on the problem setting, truck--drone synchronization may occur at customer locations \citep{schmidt2025exact} or along truck arcs \citep{li2022truck}.
An alternative architecture that reduces synchronization requirements uses fixed service hubs. Trucks transport package-equipped drones to intermediate hubs, from which the drones are dispatched to customers \citep{kim2018traveling,pinto2022point,xia2023branch}. \cite{wang2019vehicle} introduced drone stations at locations distinct from customer delivery sites. These stations support landing, recharging, and transfer operations while maintaining both active and backup drones. If a truck arrives before its assigned drone, it can continue with a backup unit, while the delayed drone becomes the reserve for a later sortie. Although this arrangement increases operational flexibility, it requires capital investment and ongoing maintenance. High delivery volumes and concentrated urban demand can help amortize these costs, whereas permanent facilities are harder to justify in rural regions with sparse demand, dispersed customers, and limited access to suitable sites with power and utilities.

\cite{kang2021exact} considered a ``wait-and-return'' strategy in which a truck carrying multiple drones parks at a waiting node, dispatches drones to nearby customers, and remains there until all drones return. This arrangement avoids moving truck--drone rendezvous and does not require permanent drone stations. Instead, visited truck stops act as temporary hubs for drone launch and retrieval. Such an arrangement is well suited to rural delivery because it can serve dispersed customers without requiring fixed drone infrastructure.

We propose the Truck-as-Hub Vehicle Routing Problem with Drones in Rural Areas (\AbbrProb), in which trucks act as temporary mobile hubs at a set of predesignated synchronization pads.
These pads are passive stopping locations, such as rural parking spaces, roadside pull-outs, farm lay-bys, or community stopping locations, and contain no permanent drone resources.
When a truck visits a pad, it may launch multiple onboard drones simultaneously to serve nearby customers and remains at the pad until all dispatched drones return before continuing its route.
The \AbbrProb{} is related to the wait-and-return setting of \cite{kang2021exact}, but differs in that drone launch and retrieval are restricted to exogenous non-customer synchronization pads instead of being permitted at arbitrary truck stops. This restriction captures rural settings in which drone operations require suitable stopping locations without relying on permanent drone facilities.
\autoref{fig:intro} shows the delivery network in panel (a) and the truck-as-hub operation at a synchronization pad in panel (b). Customers are served either directly by trucks or by drones dispatched from visited pads.

\begin{figure}[tbp]
    \centering
    \includegraphics[width=0.9\textwidth]{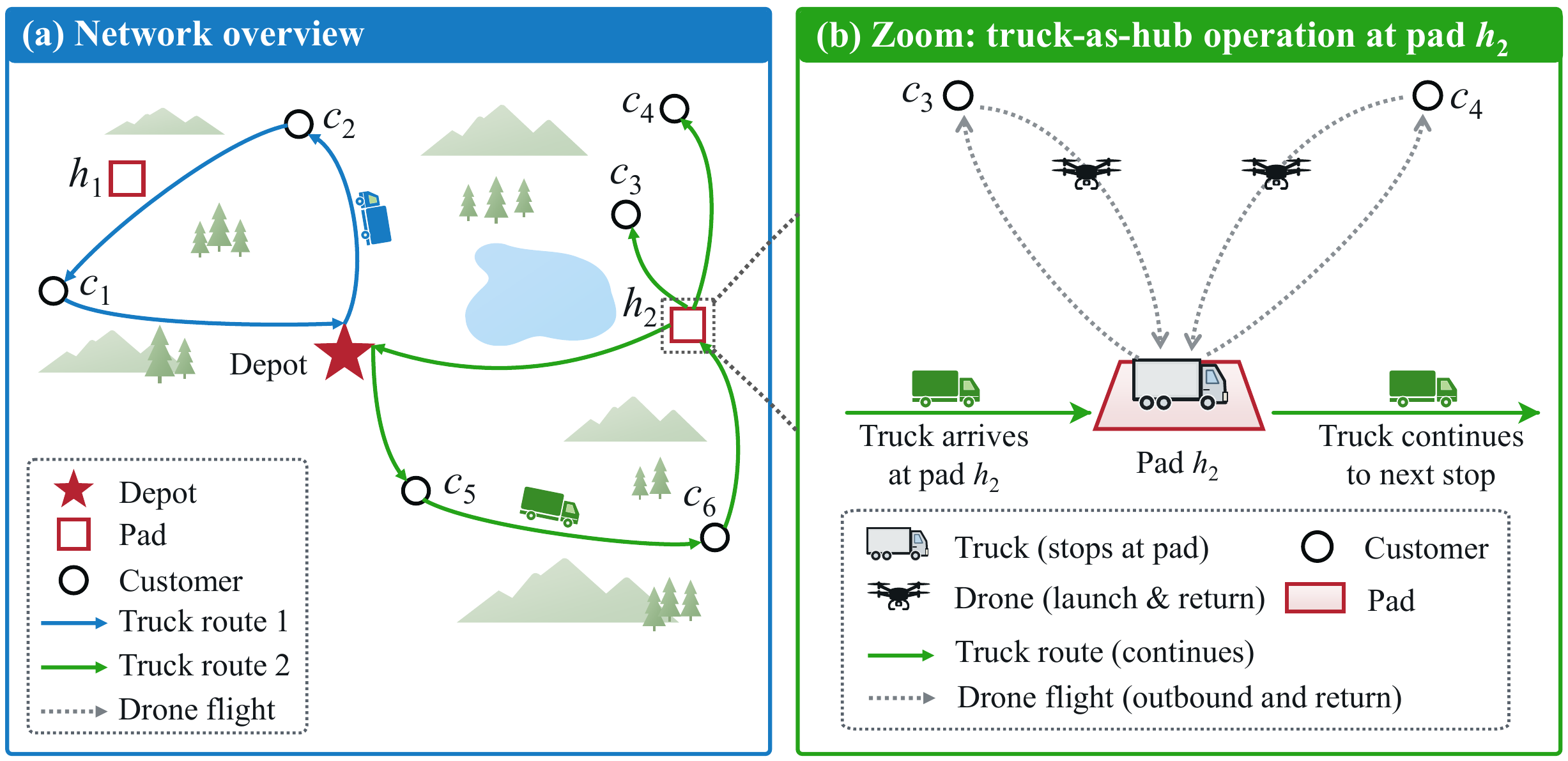}
    \caption{Illustration of the \AbbrProb{} operating concept.}
    \label{fig:intro}
\end{figure}

Another important distinction concerns the objective function. Common criteria include minimization of tour completion time or makespan \citep{wang2017vehicle,poikonen2017vehicle,agatz2018optimization} and total operating cost \citep{murray2015flying,wang2019vehicle,nguyen2022min,tamke2023vehicle}.  
Subsequent studies have broadened this scope by introducing specialized performance metrics, such as energy consumption minimization \citep{kuo2023applying}, reduction of customer service latency \citep{moshref2017using}, and maximization of customer coverage \citep{ulmer2018same}. However, these objectives do not guarantee minimal disparities in customer service levels (e.g., waiting time). Nevertheless, to the best of the authors’ knowledge, no existing work explicitly addresses service-level disparities, e.g., variability in customer waiting times, within the truck–drone last-mile delivery framework.
This shortfall is particularly acute in sparsely populated rural areas, where the long distances between hubs or depots and remote customers tend to bias service quality in favor of locations nearer the depot. A U.S. Postal Service report on low‐density delivery zones notes that ``carriers may spend a great deal of time driving down country roads to deliver to the most remote locales,'' thereby imposing longer waits on those at the periphery \citep{usps_oig_2020_package}. To address these service disparities, we formulate the \AbbrProb{} as a mixed-integer quadratic program (MIQP) with a normalized weighted objective that combines squared customer delivery delay, total truck return time, and operating cost.

The main contributions of this paper are as follows:
\begin{itemize}
    \item We introduce the Truck-as-Hub Vehicle Routing Problem with Drones in Rural Areas (\AbbrProb), in which trucks act as temporary mobile hubs at passive, predesignated synchronization pads. We formulate the problem as an arc-based mixed-integer quadratic program that jointly considers customer-level delivery delays, truck return times, and operating costs. In particular, the quadratic delay term discourages disproportionately poor service to remote customers and promotes more balanced service levels across rural regions.

    \item We construct a canonical transformed network that provides a cost-preserving elementary-path representation of each feasible truck--drone route and yields an equivalent route-space set-partitioning formulation.

    \item We develop an exact branch-price-and-cut algorithm for the route-space formulation. The algorithm combines tailored forward label-setting pricing with dominance and safe reduced-cost pruning, subset-row inequalities, and branching rules that preserve the pricing structure; parallel pricing with dynamic route-space repartitioning is used to accelerate the exact search.
\end{itemize}

\section{Problem Description}\label{sec:prob}
We consider a rural last-mile delivery system served by a fleet of trucks $\Trucks$ initially stationed at a central depot. Each truck $k\in\Trucks$ carries a fixed set of homogeneous drones $\DronesOfTruck{k}$, with $|\DronesOfTruck{k}|=\DroneFleetSize$. Let $\Drones:=\bigcup_{k\in\Trucks}\DronesOfTruck{k}$, where each drone $d\in\Drones$ is permanently assigned to truck $\TruckofDrone$. The truck and drone payload capacities are $\TruckPayload$ and $\DronePayload$, respectively, and $\DroneEndurance$ is the maximum round-trip duration of a drone sortie.
The physical network is represented by a directed graph $\Graph=(\Nodes,\Arcs)$ with node set
$
\Nodes:=\CustNodes\cup\HubNodes\cup\{\DepotSource,\DepotSink\}
$.
Here, $\CustNodes$ denotes the set of customers, $\HubNodes$ denotes the set of predesignated synchronization pads, and $\DepotSource$ and $\DepotSink$ are the source and sink copies of the central depot. These sets are pairwise disjoint. Each customer $n\in\CustNodes$ has a parcel of weight $\DemandWeight{n}>0$. A synchronization pad is a passive stopping location rather than a permanent drone facility: it contains no drones, parcels, batteries, or backup vehicles unless these resources are brought by a visiting truck.
The arc set $\Arcs$ is partitioned into the truck arc set $\TruckArcs$ and the drone sortie arc set $\DroneArcs$. Each truck arc $(i,j)\in\TruckArcs$ has travel time $\TruckTravelTime{i}{j}$. The truck arc set satisfies $\TruckArcs\subseteq(\Nodes\setminus\{\DepotSink\})\times(\Nodes\setminus\{\DepotSource\})$, so no truck arc enters $\DepotSource$ or leaves $\DepotSink$.

A drone arc $(h,j)\in\DroneArcs$, where $h\in\HubNodes$ and $j\in\CustNodes$, represents the option of serving customer $j$ by a drone launched from pad $h$. The one-way outbound flight time is denoted by $\DroneTravelTime{h}{j}$. For each $(h,j)\in\DroneArcs$, the return flight time $\DroneTravelTime{j}{h}$ is also known, although the return leg $(j,h)$ is not separately included in $\DroneArcs$. Then, the round-trip sortie time is
$
\DroneTripTime{h}{j}
:=
\DroneTravelTime{h}{j}
+
\DroneTravelTime{j}{h}.
$
A customer may be served by a drone from pad $h$ only if its parcel can be carried by both the drone and the truck and the complete sortie satisfies the drone endurance limit. Thus, $(h,j)$ is feasible only if $\DemandWeight{j}\le\DronePayload$, $\DemandWeight{j}\le\TruckPayload$, and $\DroneTripTime{h}{j}\le\DroneEndurance$.
We assume that every customer can be included in at least one complete depot-to-depot truck route, either through direct truck service or through a feasible drone sortie launched from a visited synchronization pad. Optionally, each customer $n\in\CustNodes$ may be assigned a latest allowable delivery time $\ReqServUB{n}$. For a customer without such a requirement, we set $\ReqServUB{n}=+\infty$. When a finite delivery-time limit is imposed, pre-processing verifies that at least one complete service representation satisfies it; otherwise, the instance is declared infeasible.

At a visited pad, the truck waiting time equals the longest launched round-trip sortie, or zero if none is launched; a drone-served customer's delivery time equals the truck arrival time plus its one-way outbound flight time.
For clarity, the following assumptions define the \AbbrProb: (A1) Each customer is served exactly once, either directly by a truck or by a drone launched from a synchronization pad; (A2) Each used truck departs from $\DepotSource$ and terminates at $\DepotSink$. Each drone remains assigned to the same truck throughout the planning horizon; (A3) Synchronization pads are exogenous non-customer locations, so $\HubNodes\cap\CustNodes=\emptyset$; (A4) Drones may be launched and retrieved only at synchronization pads. Each drone serves at most one customer per sortie and returns to the same truck at the same pad from which it was launched; (A5) A truck may deploy multiple drones simultaneously but cannot leave the pad until all drones launched there have returned; (A6) After returning to its truck, a drone becomes immediately available for a later sortie through battery replacement.

\section{Mathematical Formulations}\label{sec:form}

\subsection{Compact arc-based formulation}\label{arc-form}
The compact formulation uses the following decision variables:
\begin{itemize}
    \item \(\TruckUseArc{i}{j}{k}\): equals 1 if truck \(k\) traverses truck arc \((i,j)\in\TruckArcs\), and 0 otherwise.
    \item \(\DroneUseArc{h}{j}{d}\): equals 1 if drone \(d\) is launched from pad \(h\in\HubNodes\), serves customer \(j\in\CustNodes\), and returns to the same truck at pad \(h\), and 0 otherwise.
    \item \(\PadActivated{h}{k}\): equals 1 if truck \(k\) visits synchronization pad \(h\), and 0 otherwise.
    \item \(\TruckArrTime{n}{k}\ge0\): arrival time of truck \(k\) at node \(n\in\Nodes\).
    \item \(\DroneArrTime{n}{d}\ge0\): delivery time of customer \(n\) if served by drone \(d\).
    \item $\RealizedServTime{n}\ge\arriveLB{n}$: realized delivery time of customer $n\in\CustNodes$, where $\arriveLB{n}$ is a customer-specific lower bound used as the delay benchmark. Its construction is given in Appendix~\ref{sec:arr_bound}.
    \item \(\WaitTime{h}{k}\ge0\): waiting time of truck \(k\) at pad \(h\in\HubNodes\).
    \item \(\PayloadLoad{n}{k}\ge0\): cumulative real parcel weight delivered by truck \(k\) and its onboard drones after the truck reaches node \(n\).
    \item $\OrderLoad{n}{k}\ge0$: MTZ ordering variable associated with node $n$ and truck $k$.
\end{itemize}

Several constraint families below use big-$M$ terms to deactivate conditional relations when the corresponding arc is not selected, and we define a separate valid constant for each constraint family when it is introduced.
For clarity, for $n\in\Nodes$, let $\TruckOutArcs{n}$ and $\TruckInArcs{n}$ denote the sets of outgoing and incoming truck arcs, respectively. For $h\in\HubNodes$ and $n\in\CustNodes$, let $\DroneOutArcs{h}$ and $\DroneInArcs{n}$ denote the sets of drone-service arcs leaving pad $h$ and entering customer $n$, respectively.

\textbf{\emph{Truck flow and pad activation constraints}}:
\begin{subequations} \label{c:truck_flow_con}
\begin{align}
    &\sum_{(\DepotSource,j)\in \TruckOutArcs{\DepotSource}} \TruckUseArc{\DepotSource}{j}{k}
    =
    \sum_{(i,\DepotSink)\in \TruckInArcs{\DepotSink}} \TruckUseArc{i}{\DepotSink}{k},
    &&\forall k\in\Trucks,\label{c:truck_flow_con_1} \\
    &\sum_{(n,j)\in\TruckOutArcs{n}}\TruckUseArc{n}{j}{k}
    =
    \sum_{(i,n)\in\TruckInArcs{n}} \TruckUseArc{i}{n}{k},
    &&\forall n\in\CustNodes\cup\HubNodes,~k\in\Trucks, \label{c:truck_flow_con_2}\\
    &\sum_{(\DepotSource,j)\in\TruckOutArcs{\DepotSource}} \TruckUseArc{\DepotSource}{j}{k} \le 1,
    &&\forall k\in\Trucks,\label{c:truck_once}\\
    &\sum_{(n,j)\in\TruckOutArcs{n}}\TruckUseArc{n}{j}{k}\le1,
    &&\forall n\in\CustNodes\cup\HubNodes,~k\in\Trucks,\label{c:elementary-node-visit}\\
    &\PadActivated{h}{k}
    =
    \sum_{(i,h)\in\TruckInArcs{h}}\TruckUseArc{i}{h}{k},
    &&\forall h\in\HubNodes,~k\in\Trucks.
    \label{c:pad-activation}
\end{align}
\end{subequations}

Constraints $\eqref{c:truck_flow_con_1}$--$\eqref{c:truck_once}$ enforce source--sink truck flow, while \eqref{c:elementary-node-visit} limits each truck to at most one visit to any customer or synchronization pad. Constraint \eqref{c:pad-activation} links pad activation to a truck visit, thereby allowing drone service only at visited pads.

\textbf{\emph{Customer assignment constraints}}: Each customer is served exactly once, either directly by a truck or by a drone sortie. Therefore
\begin{equation}\label{c:cust_once}
    \sum_{(i,n)\in\TruckInArcs{n}} \sum_{k\in\Trucks} \TruckUseArc{i}{n}{k}
    +
    \sum_{(h,n)\in\DroneInArcs{n}} \sum_{d\in\Drones}\DroneUseArc{h}{n}{d}
    =1,
    \qquad \forall n\in\CustNodes .
\end{equation}

\textbf{\emph{Truck-carried drone launch constraints}}: for all $h\in\HubNodes,~d\in\Drones$, $\sum_{j:(h,j)\in\DroneOutArcs{h}} \DroneUseArc{h}{j}{d}\le\PadActivated{h}{\TruckofDrone}$; and for all $h\in\HubNodes,~k\in\Trucks$, $\sum_{d\in\DronesOfTruck{k}}\sum_{j:(h,j)\in\DroneOutArcs{h}} \DroneUseArc{h}{j}{d}\le\DroneFleetSize\PadActivated{h}{k}$. These constraints permit each drone to serve at most one customer during a pad visit and only when its assigned truck visits that pad; they also limit the number of simultaneous sorties to the truck's onboard drone fleet.

\textbf{\emph{Drone endurance and payload constraints}}: For every $h\in\HubNodes$, $(h,j)\in\DroneOutArcs{h}$, and $d\in\Drones$, we impose $\DroneTripTime{h}{j}\DroneUseArc{h}{j}{d}\le\DroneEndurance$ and $\DemandWeight{j}\DroneUseArc{h}{j}{d}\le\DronePayload$.

\textbf{\emph{Truck payload constraints}}:

Payload propagation counts the demand of a directly truck-served customer or the total demand assigned to drones at a visited pad. A valid big-$M$ constant for both propagation cases is
\begin{equation}\label{eq:load-big-m}
\BigMLoad
:=
\TruckPayload
+
\max\left\{
\max_{n\in\CustNodes}\DemandWeight{n},~
\max_{h\in\HubNodes}
\min\left\{
\sum_{(h,j)\in\DroneOutArcs{h}}\DemandWeight{j},~
\DroneFleetSize\DronePayload,~
\TruckPayload
\right\}
\right\}.
\end{equation}

The maximum over $h\in\HubNodes$ is interpreted as $0$ if no feasible drone sortie exists. We impose $\PayloadLoad{\DepotSource}{k}=0$ for every $k\in\Trucks$ and $\PayloadLoad{n}{k}\le\TruckPayload$ for every $n\in\Nodes$ and $k\in\Trucks$. The payload-propagation constraints are
\begin{subequations}\label{c:truck-payload}
\begin{align}
    & \PayloadLoad{n}{k}
    \ge
    \PayloadLoad{i}{k}+\DemandWeight{n}-\BigMLoad(1-\TruckUseArc{i}{n}{k}),
    && \forall n\in\CustNodes,~(i,n)\in\TruckInArcs{n},~k\in\Trucks, \label{c:truck_weight2} \\
    & \begin{aligned}
        \PayloadLoad{h}{k}
        \ge
        \PayloadLoad{i}{k}
        +\sum_{(h,j)\in\DroneOutArcs{h}}\DemandWeight{j}
        \sum_{d\in\DronesOfTruck{k}}\DroneUseArc{h}{j}{d}
        -\BigMLoad(1-\TruckUseArc{i}{h}{k}),
    \end{aligned}
    && \forall h\in\HubNodes,~(i,h)\in\TruckInArcs{h},~k\in\Trucks . \label{c:truck_weight3}
\end{align}
\end{subequations}

\textbf{\emph{MTZ-type truck-ordering constraints}}:

The truck-flow constraints may admit disconnected subtours, which we eliminate using Miller--Tucker--Zemlin (MTZ) ordering constraints \citep{miller1960integer,desrochers1991improvements}. Since an elementary truck route visits at most $|\CustNodes|+|\HubNodes|$ customer and synchronization-pad nodes, we set $\BigMOrder:=|\CustNodes|+|\HubNodes|+1$. For every $k\in\Trucks$, we impose $\OrderLoad{\DepotSource}{k}=0$ and $0\le\OrderLoad{n}{k}\le|\CustNodes|+|\HubNodes|$ for all $n\in\CustNodes\cup\HubNodes$. The ordering recursion is
\begin{equation}\label{c:truck_order_recursion}
\OrderLoad{j}{k} \ge \OrderLoad{i}{k} +1 -\BigMOrder\bigl(1-\TruckUseArc{i}{j}{k}\bigr), \quad \forall (i,j)\in\TruckArcs:j\in\CustNodes\cup\HubNodes,~k\in\Trucks.
\end{equation}

\textbf{\emph{Truck waiting time constraints}}:

Let $\MaxDroneTripTime:=\max_{(h,j)\in\DroneArcs}\DroneTripTime{h}{j}$ denote the maximum round-trip duration among feasible drone sorties, with $\MaxDroneTripTime:=0$ when $\DroneArcs=\emptyset$. The truck waiting-time constraints are
\begin{subequations}\label{c:truck_wait}
\begin{align}
&\WaitTime{h}{k}
\ge
\DroneTripTime{h}{j}
\sum_{d\in\DronesOfTruck{k}}\DroneUseArc{h}{j}{d},
&&\forall h\in\HubNodes,~k\in\Trucks,\,
(h,j)\in\DroneOutArcs{h},
\label{c:truck_wait1}\\
&\WaitTime{h}{k}
\le
\MaxDroneTripTime
\sum_{(h,j)\in\DroneOutArcs{h}}
\sum_{d\in\DronesOfTruck{k}}\DroneUseArc{h}{j}{d},
&&\forall h\in\HubNodes,~k\in\Trucks,
\label{c:truck_wait2}
\end{align}
\end{subequations}

Constraints $\eqref{c:truck_wait1}$--$\eqref{c:truck_wait2}$ set the minimum feasible $\WaitTime{h}{k}$ to the longest sortie launched at pad $h$, or to zero if no drone is launched. We also set $\WaitTime{n}{k}=0$ for $n\in\Nodes\setminus\HubNodes$ and $k\in\Trucks$.

\textbf{\emph{Time bounds for the timing constraints}}:

To formulate the conditional truck- and drone-arrival relations, we first construct a valid upper bound on the duration of any feasible truck route.
Let $\MaxTruckTravelTime:=\max_{(i,j)\in\TruckArcs}\TruckTravelTime{i}{j}$ and $\MaxDroneTravelTime:=\max_{(h,j)\in\DroneArcs}\DroneTravelTime{h}{j}$, where the latter is defined as $0$ when $\DroneArcs=\emptyset$.
Because a truck route is elementary, it contains at most $\RouteTruckArcLimit:=|\CustNodes|+|\HubNodes|+1$
truck arcs. Let $\SortedTruckArcTime{q}$ denote the $q$th largest truck travel time in the multiset $\{\TruckTravelTime{i}{j}:(i,j)\in\TruckArcs\}$, with $\SortedTruckArcTime{q}=0$ when $q>|\TruckArcs|$. For each $h\in\HubNodes$, let $\HubWaitUB{h}:=\max_{j:(h,j)\in\DroneArcs}\DroneTripTime{h}{j}$, with $\HubWaitUB{h}:=0$ when $\DroneOutArcs{h}=\emptyset$.
An elementary route uses at most $\RouteTruckArcLimit$ truck arcs and visits each synchronization pad at most once. Hence, a valid upper bound on its return time is $\RouteTimeUB := \sum_{q=1}^{\RouteTruckArcLimit}\SortedTruckArcTime{q} + \sum_{h\in\HubNodes}\HubWaitUB{h}$.

\textbf{\emph{Arrival and realized service-time constraints}}:

We impose $\TruckArrTime{\DepotSource}{k}=0$ for every $k\in\Trucks$ and $\RealizedServTime{n}\ge\arriveLB{n}$ for every $n\in\CustNodes$ and
\begin{subequations}\label{c:realized_group}
\begin{align}
&\TruckArrTime{j}{k}
\ge
\TruckArrTime{i}{k}
+\WaitTime{i}{k}
+\TruckTravelTime{i}{j}
-\BigMTime(1-\TruckUseArc{i}{j}{k}),
&&\forall (i,j)\in\TruckArcs,~k\in\Trucks,
\label{c:truck_arrive1a}\\
&\TruckArrTime{j}{k}
\le
\TruckArrTime{i}{k}
+\WaitTime{i}{k}
+\TruckTravelTime{i}{j}
+\BigMTime(1-\TruckUseArc{i}{j}{k}),
&&\forall (i,j)\in\TruckArcs,~k\in\Trucks,
\label{c:truck_arrive1b}\\
&\DroneArrTime{j}{d}
\ge
\TruckArrTime{h}{\TruckofDrone}
+\DroneTravelTime{h}{j}
-\BigMTime(1-\DroneUseArc{h}{j}{d}),
&&\forall h\in\HubNodes,~(h,j)\in\DroneOutArcs{h},~d\in\Drones,
\label{c:drone_arrive1}\\
&\DroneArrTime{j}{d}
\le
\TruckArrTime{h}{\TruckofDrone}
+\DroneTravelTime{h}{j}
+\BigMTime(1-\DroneUseArc{h}{j}{d}),
&&\forall h\in\HubNodes,~(h,j)\in\DroneOutArcs{h},~d\in\Drones,
\label{c:drone_arrive2}\\
&\RealizedServTime{n}
\ge
\TruckArrTime{n}{k}
-\BigMTime\left(
1-\sum_{(i,n)\in\TruckInArcs{n}}\TruckUseArc{i}{n}{k}
\right),
&&\forall n\in\CustNodes,~k\in\Trucks,
\label{c:realized-truck}\\
&\RealizedServTime{n}
\ge
\DroneArrTime{n}{d}
-\BigMTime\left(
1-\sum_{(h,n)\in\DroneInArcs{n}}\DroneUseArc{h}{n}{d}
\right),
&&\forall n\in\CustNodes,~d\in\Drones.
\label{c:realized-drone}
\end{align}
\end{subequations}

Constraints $\eqref{c:truck_arrive1a}$--$\eqref{c:truck_arrive1b}$ propagate truck arrival times along selected truck arcs. The arrival time at the head of an arc equals the arrival time at its tail plus any waiting time at the tail and the corresponding truck travel time. Constraints $\eqref{c:drone_arrive1}$--$\eqref{c:drone_arrive2}$ compute the arrival time of a drone-served customer using the truck's arrival time at the launch pad and the one-way drone travel time. Thus, the one-way time $\DroneTravelTime{h}{j}$ determines customer service time, whereas the round-trip time $\DroneTripTime{h}{j}$ determines the truck waiting time in $\eqref{c:truck_wait1}$.

We now specify the deactivation constant $\BigMTime$ used in these constraints. When $\TruckUseArc{i}{j}{k}=0$, the right-hand side of $\eqref{c:truck_arrive1a}$ can contain a truck arrival time of at most $\RouteTimeUB$, a pad waiting time of at most $\MaxDroneTripTime$, and a truck travel time of at most $\MaxTruckTravelTime$. Therefore, a valid deactivation bound for the truck-arrival relations is $\RouteTimeUB+\MaxDroneTripTime+\MaxTruckTravelTime$.
Similarly, when $\DroneUseArc{h}{j}{d}=0$, constraints $\eqref{c:drone_arrive1}$--$\eqref{c:drone_arrive2}$ should impose no relationship between $\DroneArrTime{j}{d}$ and $\TruckArrTime{h}{\TruckofDrone}$. Since $\TruckArrTime{h}{\TruckofDrone}\le\RouteTimeUB$ and $\DroneTravelTime{h}{j}\le\MaxDroneTravelTime$, the constant $\RouteTimeUB+\MaxDroneTravelTime$ is sufficient to deactivate these constraints.
These bounds also deactivate \eqref{c:realized-truck}--\eqref{c:realized-drone} when the corresponding vehicle does not serve the customer. Hence, a common valid constant is $\BigMTime:=\RouteTimeUB+\max\{\MaxDroneTripTime+\MaxTruckTravelTime,\MaxDroneTravelTime\}$.

\textbf{\emph{Variable domains and explicit bounds}}:
The variables $\TruckUseArc{i}{j}{k}$, $\DroneUseArc{h}{j}{d}$, and $\PadActivated{h}{k}$ are binary over their respective truck arcs, drone arcs, and pad--truck pairs. For each customer $n\in\CustNodes$, define $\EffServUB{n}:=\min\{\ReqServUB{n},\ServTimeUB{n}\}$, where $\ReqServUB{n}$ is the prescribed latest delivery time and $\ServTimeUB{n}$ is the valid delivery-time upper bound derived below from the route-duration and travel-time bounds. The continuous variables satisfy $0\le\TruckArrTime{n}{k}\le\RouteTimeUB$, $0\le\DroneArrTime{n}{d}\le\RouteTimeUB+\MaxDroneTravelTime$, $\arriveLB{n}\le\RealizedServTime{n}\le\EffServUB{n}$, $0\le\WaitTime{h}{k}\le\HubWaitUB{h}$, $0\le\PayloadLoad{n}{k}\le\TruckPayload$, and $0\le\OrderLoad{n}{k}\le|\CustNodes|+|\HubNodes|$ over their corresponding index sets.

Let $\TruckCost$ denote the fixed cost of deploying a truck and $\DroneCost$ the cost of a drone sortie. The objective combines squared delivery delay $\ObjDelay:=\sum_{n\in\CustNodes}(\RealizedServTime{n}-\arriveLB{n})^2$, total truck return time $\ObjReturn:=\sum_{k\in\Trucks}\TruckArrTime{\DepotSink}{k}$, and operating cost $\ObjCost$, defined as
\begin{equation}\label{obj:cost-term}
\ObjCost
:=
\TruckCost
\sum_{k\in\Trucks}
\sum_{(\DepotSource,j)\in\TruckOutArcs{\DepotSource}}
\TruckUseArc{\DepotSource}{j}{k}
+
\DroneCost
\sum_{d\in\Drones}
\sum_{n\in\HubNodes}
\sum_{(n,j)\in\DroneOutArcs{n}}
\DroneUseArc{n}{j}{d}.
\end{equation}

For customer $n$, the delivery delay is $\RealizedServTime{n}-\arriveLB{n}$. We penalize this nonnegative delay quadratically, so later service incurs a larger marginal penalty. This encourages more comparable service levels across geographically dispersed customers, consistent with convex delay penalties used in related routing models \citep{russell2008vehicle,tacs2014vehicle}.

To derive the upper bound $\ServTimeUB{n}$, let $\ShortestTruckTime{i}{j}$ be the shortest directed truck travel time in $(\Nodes,\TruckArcs)$, using $\TruckTravelTime{u}{v}$ as arc lengths and setting $\ShortestTruckTime{i}{j}=+\infty$ when $j$ is unreachable from $i$. These shortest-path values relax the route's operating constraints and therefore provide lower bounds on the corresponding truck travel times. Using the route-duration bound $\RouteTimeUB$, define $\TruckServUB{n}:=\RouteTimeUB-\ShortestTruckTime{n}{\DepotSink}$ when $\ShortestTruckTime{n}{\DepotSink}<+\infty$, and $\TruckServUB{n}:=-\infty$ otherwise. For drone service, let $\mathcal H_n^{\mathrm{UB}}:=\{h\in\HubNodes:(h,n)\in\DroneArcs,\,\ShortestTruckTime{h}{\DepotSink}<+\infty\}$ and define $\DroneServUB{n}:=\max_{h\in\mathcal H_n^{\mathrm{UB}}}\{\RouteTimeUB-\DroneTravelTime{n}{h}-\ShortestTruckTime{h}{\DepotSink}\}$, where $\max\emptyset:=-\infty$. Finally, the mode-independent bound is $\ServTimeUB{n}:=\max\{\arriveLB{n},\TruckServUB{n},\DroneServUB{n}\}$.

Next we derive valid lower and upper bounds for each term to normalize its contribution to the weighted objective. Because every customer can be served by at least one complete feasible route, as assumed in Section~\ref{sec:prob}, at least one of $\TruckServUB{n}$ and $\DroneServUB{n}$ is finite. Therefore, $\ServTimeUB{n}$ is finite and $\EffServUB{n}:=\min\{\ServTimeUB{n},\ReqServUB{n}\}$ is also finite. Given that $\arriveLB{n}\le\RealizedServTime{n}\le\EffServUB{n}$, valid delay bounds are $\delayLB:=0$ and $\delayUB:=\sum_{n\in\CustNodes}(\EffServUB{n}-\arriveLB{n})^2$.
Without loss of optimality, every used truck serves at least one customer, so $\MaxUsedTruckNum:=\min\{|\Trucks|,|\CustNodes|\}$. A payload-based lower bound is $\MinTruckNum:=\max\{1,\lceil \sum_{n\in\CustNodes}\DemandWeight{n}/\TruckPayload\rceil\}$; if $\MinTruckNum>|\Trucks|$, the instance is infeasible. Let $\TruckArrTimeLB{\DepotSink}:=\min_{q\in\CustNodes\cup\HubNodes}\{\ShortestTruckTime{\DepotSource}{q}+\ShortestTruckTime{q}{\DepotSink}\}$. Then $\returnLB:=\MinTruckNum\TruckArrTimeLB{\DepotSink}$ and $\returnUB:=\MaxUsedTruckNum\RouteTimeUB$.
For operating cost, define $\MandatoryDroneCust:=\{n\in\CustNodes:\ShortestTruckTime{\DepotSource}{n}=+\infty \textnormal{ or } \ShortestTruckTime{n}{\DepotSink}=+\infty\}$ and $\MinDroneFlightNum:=|\MandatoryDroneCust|$. Hence, $\costLB:=\MinTruckNum\TruckCost+\MinDroneFlightNum\DroneCost$ and $\costUB:=\MaxUsedTruckNum\TruckCost+|\CustNodes|\DroneCost$.

The weights $\WeightDelay$, $\WeightReturn$, and $\WeightCost$ specify the relative importance of the normalized squared-delay, total truck-return-time, and operating-cost terms, respectively. Then, the normalization coefficients are $\ObjCoeffDelay:=\WeightDelay/(\delayUB-\delayLB)$, $\ObjCoeffReturn:=\WeightReturn/(\returnUB-\returnLB)$, and $\ObjCoeffCost:=\WeightCost/(\costUB-\costLB)$; if any denominator is zero, the corresponding coefficient is defined as zero.
The resulting weighted objective is
\begin{equation}\label{obj:normalized}
\min\quad
\ObjCoeffDelay(\ObjDelay-\delayLB)
+
\ObjCoeffReturn(\ObjReturn-\returnLB)
+
\ObjCoeffCost(\ObjCost-\costLB),
\end{equation}
where $\WeightDelay,\WeightReturn,\WeightCost>0$ and $\WeightDelay+\WeightReturn+\WeightCost=1$.

\subsection{Path-based formulation for the \AbbrProb}\label{sec:path-form}

The arc-based formulation is useful for defining the operating rules, but its LP relaxation is weak, and the number of arc-indexed variables grows quickly with the number of trucks and drones. We therefore reformulate the problem in route space. A route column represents the complete operation of one truck and its onboard drone fleet.
For any route \(r\) for one truck and its drone fleet, let
$
\TruckRoute{r}=(v_0^r,v_1^r,\ldots,v_{\ell_r}^r)
$
denote its physical truck projection, namely, the ordered sequence of physical nodes visited by the truck. Drone-served customers are not included in this sequence. The route begins and ends at the depot copies, so \(v_0^r=\DepotSource\), \(v_{\ell_r}^r=\DepotSink\), and \(v_q^r\in\CustNodes\cup\HubNodes\) for \(q=1,\ldots,\ell_r-1\). We require the physical truck projection to be elementary:
$
\sum_{q=1}^{\ell_r-1}\mathbf{1}\{v_q^r=i\}\le 1,
$
for any $i\in\CustNodes\cup\HubNodes$.

Let \(\RouteSet\) denote the finite set of feasible routes with elementary physical truck projections that serve at least one customer by either truck or drone. We refer to such routes as nonempty routes.
A route $r\in\RouteSet$ consists of an elementary truck path $\TruckRoute{r}$ from $\DepotSource$ to $\DepotSink$ and a feasible drone-service block $\DroneBlock{h}{r}$ at each visited pad $h$. Each block $\DroneBlock{h}{r}$ contains the customers served by drones launched simultaneously from the truck at $h$. We set $\DroneBlock{h}{r}=\emptyset$ when route $r$ launches no drone from pad $h$.
For each customer $n$ served by route $r$, $\RouteArrTime{n}{r}$ denotes its delivery time, and $\RouteArrTime{\DepotSink}{r}$ denotes the truck's return time. At each visited pad $h$, $\TruckArrTime{h}{r}$ and $\WaitTime{h}{r}$ denote the truck's arrival and waiting times, respectively. Every served customer must satisfy $\RouteArrTime{n}{r}\le\EffServUB{n}$.

Let \(\RouteVisitNode{r}{n}\) equal 1 if customer \(n\) is served by route \(r\), either by the truck or by one of its drones, and 0 otherwise. Define $\RouteServedSet{r}:=\{n\in\CustNodes:\RouteVisitNode{r}{n}=1\}$, which includes both truck- and drone-served customers.
The route cost is defined as
\begin{equation}\label{eq:route-cost-def}
\RouteCost{r}
=
\ObjCoeffDelay
\sum_{n\in\RouteServedSet{r}}
\left(\RouteArrTime{n}{r}-\arriveLB{n}\right)^2
+
\ObjCoeffReturn\RouteArrTime{\DepotSink}{r}
+
\ObjCoeffCost
\left(
\TruckCost
+
\DroneCost
\sum_{h\in\HubNodes}
\left|\DroneBlock{h}{r}\right|
\right),
\end{equation}

Let $\RouteSelect{r}\in\{0,1\}$ indicate whether route $r$ is selected. The \AbbrProb{} is then the set-partitioning problem that minimizes $\sum_{r\in\RouteSet}\RouteCost{r}\RouteSelect{r}$ subject to $\sum_{r\in\RouteSet}\RouteVisitNode{r}{n}\RouteSelect{r}=1$ for all $n\in\CustNodes$ and $\sum_{r\in\RouteSet}\RouteSelect{r}\le|\Trucks|$. The covering constraints ensure that each customer is served exactly once, while the latter constraint enforces the truck-fleet limit; all remaining operational requirements are embedded in the definition of a feasible route.
The full normalized objective \eqref{obj:normalized} equals the route-space objective above plus the constant $\RouteObjShift:=-\ObjCoeffDelay\delayLB-\ObjCoeffReturn\returnLB-\ObjCoeffCost\costLB$.

\section{Branch-Price-and-Cut Algorithm}\label{sec:BPC}

The set-partitioning formulation contains exponentially many routes, so we solve it by branch-price-and-cut (BPC). At branch-and-bound node $m$, let $\BranchFeasibleRoutes{m}\subseteq\RouteSet$ denote the routes satisfying all inherited branching decisions and $\RouteSubset{m}\subseteq\BranchFeasibleRoutes{m}$ the columns currently generated. The restricted master problem (RMP) is the LP relaxation of this set-partitioning model restricted to $\RouteSubset{m}$ and augmented with active subset-row (SR) inequalities.
For each triple $\pi\in\CandidateSRset{m}:={\pi\subseteq\CustNodes:|\pi|=3}$, define $\sigma_{r,\pi}:=\left\lfloor\sum_{n\in\pi}\RouteVisitNode{r}{n}/2\right\rfloor$. For the active set $\ActiveSRset{m}\subseteq\CandidateSRset{m}$, the RMP includes $\sum_{r\in\RouteSubset{m}}\sigma_{r,\pi}\RouteSelect{r}\le1$ for $\pi\in\ActiveSRset{m}$. Let $\mu_n\in\mathbb R$, $\kappa\le0$, and $\nu_\pi\le0$ denote the dual variables associated with the customer-covering, truck-fleet, and active SR constraints, respectively.
Standard pricing generates negative-reduced-cost routes from $\BranchFeasibleRoutes{m}$, while separation adds violated inactive SR inequalities and triggers repricing under the updated dual solution. Farkas pricing is used when the RMP is infeasible. The process continues until neither pricing nor separation adds a route or inequality, after which the node is fathomed or branched on. For brevity, we write $\SRset$ for $\ActiveSRset{m}$ below. Pricing is formulated as an elementary shortest path problem with resource constraints (ESPPRC) on the transformed network \citep{dror1994note}.

\subsection{Network transformation}\label{sec:net-trans}
The transformed network encodes both the truck path and the drone-service blocks within the same source--sink path, while timing, waiting, payload, and cost are tracked through path resources. This construction provides the path representation required for the label-setting pricing algorithm.

For each $h\in\HubNodes$, define $\DroneAccessNodes{h}:=\{j\in\CustNodes:(h,j)\in\DroneArcs\}$ and $\TruckAccessNodes{h}:=\{u\in\Nodes:(h,u)\in\TruckArcs\}$, the customers accessible by drone from $h$ and the physical nodes reachable by truck from $h$, respectively.
For each pad $h$, we order the customers in $\DroneAccessNodes{h}$ by increasing customer index and write $i\prec_h j$ whenever customer $i$ precedes customer $j$. This ordering gives each set of simultaneous drone sorties a unique path representation.

To distinguish drone service from direct truck service, each customer that is drone-accessible from at least one synchronization pad is duplicated once. Specifically, as shown in \autoref{fig:net-trans}, $\CustNodesDuplicate:=\{\Duplicate{j}:j\in\bigcup_{h\in\HubNodes}\DroneAccessNodes{h}\}$.
Visiting the original node $j$ represents direct truck service, whereas visiting $\Duplicate{j}$ represents drone service. For each $\Duplicate{j}\in\CustNodesDuplicate$, define $\Origin{\Duplicate{j}}:=j$. A duplicated node does not identify its launch pad; instead, the launch pad is the physical pad immediately preceding the corresponding duplicated-node block.
Using this duplicated-node set, define the transformed network as $\TransGraph=(\Duplicate{\Nodes},\Duplicate{\Arcs})$, where $\Duplicate{\Nodes}:=\Nodes\cup\CustNodesDuplicate$ and $\Duplicate{\Arcs}:=\TruckArcs\cup\CustHubArcsDuplicate\cup\CustDupCustDupArcs\cup\CustDupCustArcs$.
The additional arc families are $\CustHubArcsDuplicate:=\{(h,\Duplicate{j}):h\in\HubNodes,j\in\DroneAccessNodes{h}\}$, $\CustDupCustDupArcs:=\{(\Duplicate{i},\Duplicate{j}):\exists h\in\HubNodes,i,j\in\DroneAccessNodes{h},i\prec_h j\}$, and $\CustDupCustArcs:=\{(\Duplicate{i},u):\exists h\in\HubNodes,i\in\DroneAccessNodes{h},u\in\TruckAccessNodes{h},u\neq i\}$.

An arc $(h,\Duplicate{j})\in\CustHubArcsDuplicate$ enters a drone-service block at pad $h$ and selects customer $j$ for drone service. An arc $(\Duplicate{i},\Duplicate{j})\in\CustDupCustDupArcs$ adds customer $j$ to the same block. An arc $(\Duplicate{i},u)\in\CustDupCustArcs$ exits the drone-service block and resumes the truck route; during label extension, $u$ must be truck-accessible from the active pad.
Accordingly, a transformed path segment of the form
$
h
\rightarrow
\Duplicate{j_1}
\rightarrow
\cdots
\rightarrow
\Duplicate{j_s}
\rightarrow
u
$
represents a truck visit to pad $h$, the simultaneous drone service of customers $j_1,\ldots,j_s$, and the subsequent truck movement from $h$ to $u$. The physical pad $h$ preceding the duplicated-node block identifies the launch pad for all customers in that block. The sequential appearance of the duplicated nodes is purely an encoding convention and does not imply that the corresponding drone sorties are performed sequentially.

\begin{figure}[tbp]
    \centering
    \includegraphics[width=0.9\linewidth]{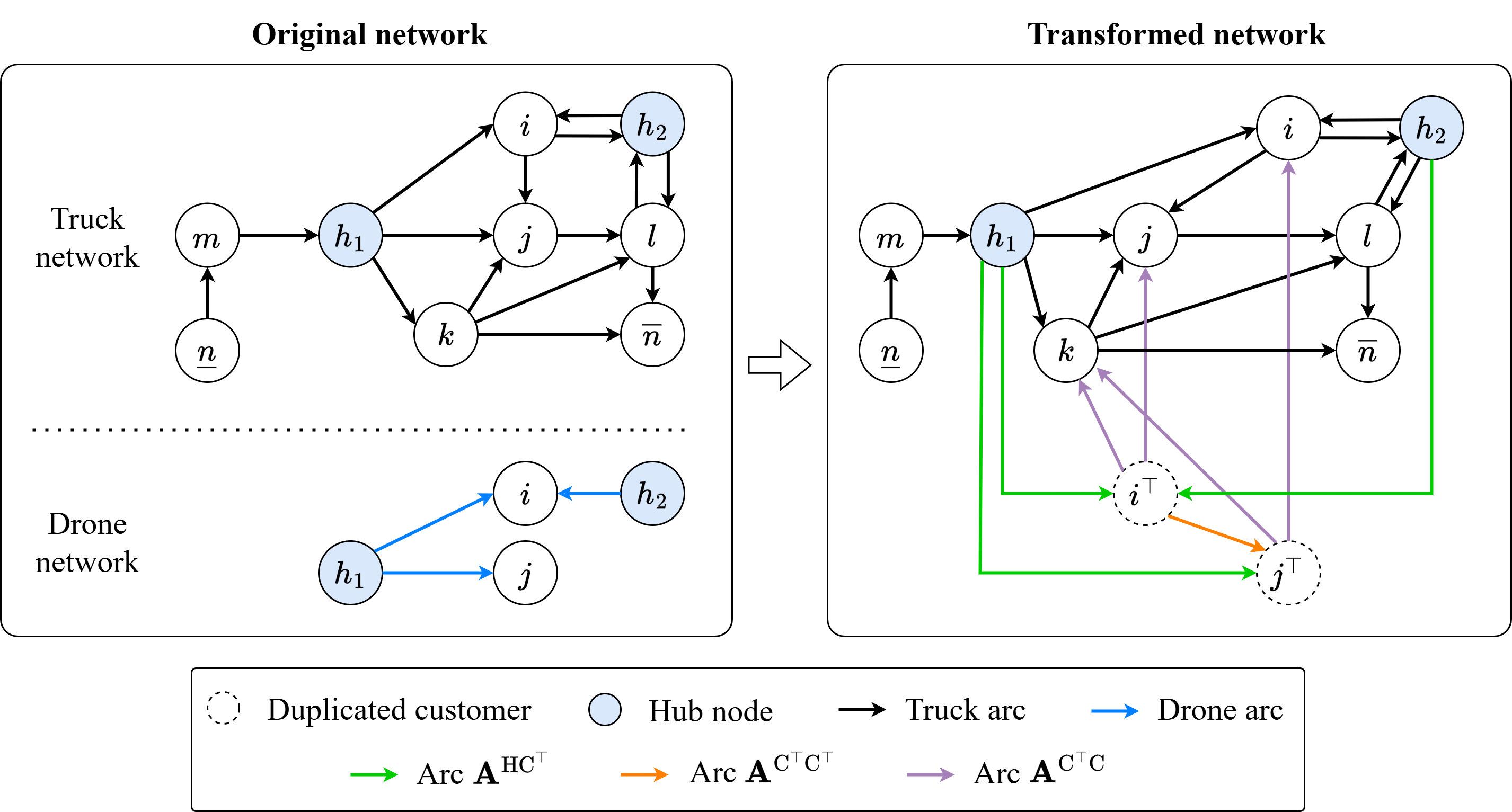}
    \caption{Illustration of the transformed network.}
    \label{fig:net-trans}
\end{figure}

Let $\TravelTime{i}{j}$ denote the time attribute of $(i,j)\in\Duplicate{\Arcs}$. For a truck arc, $\TravelTime{i}{j}:=\TruckTravelTime{i}{j}$. For an entry arc $(h,\Duplicate{j})\in\CustHubArcsDuplicate$, $\TravelTime{h}{\Duplicate{j}}:=\DroneTripTime{h}{j}$. For a transformed path prefix whose active pad is $h$, a within-block arc satisfies $\TravelTime{\Duplicate{i}}{\Duplicate{j}}:=\DroneTripTime{h}{j}$, while an exit arc satisfies $\TravelTime{\Duplicate{i}}{u}:=\TruckTravelTime{h}{u}$. In the label-setting algorithm, this pad is maintained by the active-pad resource. Drone-block arc attributes record individual sortie durations and are not accumulated on the truck clock; instead, the truck waiting resource is their maximum, while customer $j$ is served at $\TruckArrTime{h}{r}+\DroneTravelTime{h}{j}$.

\begin{theorem}[Cost-preserving route encoding] \label{thm:cost-preserving-encoding}
Under Assumptions~(A1)--(A6) and the elementary requirement, every feasible truck--drone route has a unique canonical source--sink path representation in $\TransGraph$ when identical drones are treated as indistinguishable. This representation preserves route feasibility and cost.
\end{theorem}
\begin{proof} See Appendix~\ref{app:proof:thm-cost-preserving-encoding}. \end{proof}

\subsubsection{Complexity effects of the network transformation}
\label{sec:complexity}
The transformation enables path-based pricing but introduces additional nodes and arcs that may increase label extensions. We therefore quantify its size in terms of pad accessibility and arc sharing.

Let $\NumCust:=|\CustNodes|$, $\NumHub:=|\HubNodes|$, $\DroneAccessNum{h}:=|\DroneAccessNodes{h}|$, and $\TruckAccessNum{h}:=|\TruckAccessNodes{h}|$. The number of distinct drone-accessible customers is $Q:=|\bigcup_{h\in\HubNodes}\DroneAccessNodes{h}|$. Since each is duplicated once, $|\Duplicate{\Nodes}|=|\Nodes|+Q\le|\Nodes|+\min\{\NumCust,\sum_{h\in\HubNodes}\DroneAccessNum{h}\}$. 
Entry arcs $(h,\Duplicate{j})$ are distinct for different pads because the pad is one of their endpoints. Each within-block arc $(\Duplicate{i},\Duplicate{j})$ or exit arc $(\Duplicate{i},u)$ is included only once in the transformed network, even when it can arise from several pads.
Under the common customer ordering in Section~\ref{sec:net-trans}, the arc set therefore satisfies
\begin{equation}\label{eq:transformed_arc_general_bound}
|\Duplicate{\Arcs}| \le |\TruckArcs| + \sum_{h\in\HubNodes}\DroneAccessNum{h} + \min\left\{\binom{Q}{2},\,\sum_{h\in\HubNodes}\binom{\DroneAccessNum{h}}{2}\right\} + \min\left\{Q|\Nodes|,\,\sum_{h\in\HubNodes}\DroneAccessNum{h}\TruckAccessNum{h}\right\}.
\end{equation}

The three added terms bound the entry, within-block, and exit arc families, respectively. Each minimum accounts for arc sharing by combining the pad-wise count with a bound on the number of distinct endpoint pairs. If $\DroneAccessNum{h}$ and $\TruckAccessNum{h}$ are uniformly bounded, the number of additional arcs grows at most linearly with $\NumHub$. In dense instances, the $Q^2$ and $Q|\Nodes|$ terms may dominate even with the canonical ordering introduced in Section~\ref{sec:net-trans}. Having quantified the size of the transformed network, we next describe the pricing-compatible branching rules used in BPC.

\subsection{Branching strategies}\label{sec:branching}
\subsubsection{Customer-pair branching.}
For two distinct customers $p,q\in\CustNodes$, define the same-route flow
$
\SameRouteFlow{p}{q}
:=
\sum_{r\in\RouteSubset{m}}
\RouteVisitNode{r}{p}\RouteVisitNode{r}{q}\RouteSelect{r},
$
which, in an integer solution, equals 1 if $p$ and $q$ are served by the same selected route, and 0 otherwise. If $0<\SameRouteFlow{p}{q}<1$, we create two child nodes:
$
\BranchTogether{pq}:=
    \left\{
    r\in\BranchFeasibleRoutes{m}:
    \RouteVisitNode{r}{p}=\RouteVisitNode{r}{q}
    \right\},
$
and
$
\BranchSeparate{pq}
    :=
    \left\{
    r\in\BranchFeasibleRoutes{m}:
    \RouteVisitNode{r}{p}+\RouteVisitNode{r}{q}\le 1
    \right\}.
$
Here, $\BranchTogether{pq}$ denotes the routes that must serve either both $p$ and $q$ or neither of them, whereas $\BranchSeparate{pq}$ represents the routes serving at most one of them.
The two route sets may overlap on routes that serve neither $p$ nor $q$. Nevertheless, the corresponding integer solution spaces are disjoint and exhaustive: in every integer solution, the unique selected route covering $p$ either also covers $q$, or $q$ is covered by a different selected route.
During pricing, the branch $\BranchTogether{pq}$ is enforced at completion by accepting only routes that represent both customers or neither customer. The branch $\BranchSeparate{pq}$ is enforced locally by forbidding an extension that represents $q$ after $p$ has already been represented, and vice versa. Moreover, in the branch $\BranchTogether{pq}$, a partial label that has included $p$ but not $q$ may be pruned if $q$ can no longer be served feasibly, e.g., because its service-time limit or payload requirement cannot be satisfied, or inherited branch restrictions prohibit all remaining service options.

\subsubsection{Customer-conditioned transformed-arc branching.}
If no fractional customer-pair flow exists while the RMP solution remains fractional, we branch on the transformed paths. For each route $r\in\RouteSet$ and arc $(i,j)\in\Duplicate{\Arcs}$, let $\RouteTravelArc{r}{i}{j}$ be the binary indicator that equals 1 if route $r$ uses arc $(i,j)$, and equals 0 otherwise.
Then, for a customer $p\in\CustNodes$ and an arc $(i,j)\in\Duplicate{\Arcs}$, define the customer-conditioned transformed-arc flow
$
\TransArcFlow{p}{i}{j}
:=
\sum_{r\in\RouteSubset{m}}
\RouteVisitNode{r}{p}
\RouteTravelArc{r}{i}{j}
\RouteSelect{r},
$
to be the total LP weight of routes that both serve customer $p$ and traverse transformed arc $(i,j)$. Because $p$ is covered exactly once, $\TransArcFlow{p}{i}{j}$ equals $1$ in an integer solution if the selected route covering $p$ traverses $(i,j)$, and $0$ otherwise.
If $0<\TransArcFlow{p}{i}{j}<1$, we create the two child nodes
$
\BranchArcDown{p}{i}{j}
:=
\left\{
r\in\BranchFeasibleRoutes{m}:
\RouteVisitNode{r}{p}
+
\RouteTravelArc{r}{i}{j}
\le 1
\right\}
$
and
$
\BranchArcUp{p}{i}{j}
:=
\left\{
r\in\BranchFeasibleRoutes{m}:
\RouteVisitNode{r}{p}
\le
\RouteTravelArc{r}{i}{j}
\right\},
$
where $\BranchArcDown{p}{i}{j}$ forbids the route covering $p$ from traversing arc $(i,j)$, while $\BranchArcUp{p}{i}{j}$ requires that route to traverse $(i,j)$. Routes not serving $p$ remain admissible in both children.
This branching rule guarantees completeness because any fractional route serving customer $p$ implies another positive route also serving $p$. Since distinct routes have unique path representations in the transformed network, they differ on at least one transformed arc $(i,j)$, yielding $0<\TransArcFlow{p}{i}{j}<1$.

\subsection{Pricing problem and label setting algorithm}\label{sec:label-setting}

At node $m$, pricing minimizes $\overline c_r=\RouteCost{r}-\sum_{n\in\CustNodes}\RouteVisitNode{r}{n}\mu_n-\sum_{\SRsubset\in\ActiveSRset{m}}\sigma_{r,\SRsubset}\nu_{\SRsubset}-\kappa$ over nonempty routes $r\in\BranchFeasibleRoutes{m}$.
A forward label records its partial path, covered customers, visited physical nodes, payload, timing, active pad, drone count, SR counters, branch state, and reduced cost.
Dominance rules compare feasible completions, including labels with different transformed endpoints but the same physical truck location, to discard labels whose every feasible completion can be matched by another label at no greater reduced cost.
To avoid extending labels that cannot generate a negative-reduced-cost route, we compute a safe lower bound on the reduced cost of any feasible completion. A label is discarded if this bound is nonnegative.
Appendices~\ref{app:label-setting-details} and~\ref{app:accel-full} give the exact conditions, recursions, and proofs' references.

\subsubsection{Farkas pricing.}\label{sec:farkas-pricing}

For an infeasible RMP, the same feasible-extension machinery searches for a route violating its Farkas certificate. A violating route is inserted; the full node master is declared infeasible only when no such route exists. Farkas reduced costs contain extreme rays, requiring the separate dominance and pruning conditions in Appendix~\ref{app:farkas-details}.

\subsubsection{Parallel and primal procedures.}\label{sec:accel-summary}

Pricing initially partitions routes by their first physical node and assigns the subsets to different parallel processes using a workload score. Dynamic splitting transfers complete open-label states to idle processes; all processes use the same dual solution, cuts, and branching restrictions. Closure requires all tasks to be exhausted or safely closed, with no workload transfer pending. For the primal procedures, a time-limited solve of the compact formulation supplies initial route columns when available. During the BPC search, the integer set-partitioning model is also solved over the accumulated branch-feasible route pool to obtain improved incumbent solutions. These scheduling and primal procedures are specified in Appendix~\ref{app:accel-full}.

\section{Computational Experiments}\label{sec:exp}
We evaluate the proposed BPC on 24 synthetic instances generated under two customer spatial layouts, purely sparse (PS) and purely clustered (PC), and three profiles $(|\CustNodes|,|\Trucks|,|\HubNodes|,\DroneFleetSize)\in\{(5,2,2,4),(20,4,3,4),(30,6,3,4)\}$. These profiles are referred to as small, medium, and large, respectively. For each layout--profile combination, four instances are generated, yielding 24 instances in total. Every instance is solved by both the proposed BPC algorithm and Gurobi applied directly to the compact arc-based MIQP using identical instance data and objective parameters. The detailed instance-generation procedure and computational settings are provided in Appendix~\ref{app:test-instance-generation}.
The implementation is written in Python~3.11.7 and uses Gurobi~12.0.1 on a Windows workstation with an AMD Ryzen~5~5600 processor. The direct MIQP uses Gurobi's default settings with multi-threading enabled, whereas BPC uses 12 parallel pricing processes.

\subsection{Overall Computational Performance}
\autoref{tab:small-scale-results}-\autoref{tab:large-scale-results} report the results for the three instance scales. ``Inst.'' and ``Patt.'' denote the instance and customer pattern (PS: purely sparse; PC: purely clustered), and $|\Duplicate{\Arcs}|$ is the number of transformed-network arcs. ``MIQP'' and ``BPC'' refer to Gurobi applied to the compact arc-based MIQP and the proposed branch-price-and-cut algorithm, respectively. ``Inc.'', ``LB'', ``Gap'', and ``Time'' denote the final incumbent, lower bound, relative optimality gap, and execution time (s). ``Root UB'' and ``Root LB'' are the root warm-start incumbent and root-node lower bound. For BPC, ``Cols.\ W/S/F'' gives the numbers of columns added by warm start, standard pricing, and Farkas pricing; ``\# SR'' counts subset-row inequalities; ``Br.\ P/A'' counts customer-pair and customer-conditioned transformed-arc branching decisions; and ``\# Split'' and ``\# Nodes'' count pricing-task repartitioning events and processed branch-and-bound nodes. A dash indicates an inapplicable or unavailable statistic.

\begin{table}[htbp]
\centering
\scriptsize
\renewcommand{\arraystretch}{0.9}
\setlength{\tabcolsep}{2pt}
\caption{Solution Details of Small-Scale Instances
$(|\CustNodes|,|\Trucks|,|\HubNodes|,\DroneFleetSize)=(5,2,2,4)$}
\label{tab:small-scale-results}
\begin{tabular*}{\textwidth}{
@{\extracolsep{\fill}}
cccccrrrcrccccr
@{}}
\toprule
\multirow[c]{2}{*}{Inst.}
& \multirow[c]{2}{*}{Patt.}
& \multirow[c]{2}{*}{$|\Duplicate{\Arcs}|$}
& \multirow[c]{2}{*}{Method}
& \multirow[c]{2}{*}{Inc.}
& \multicolumn{1}{c}{\multirow[c]{2}{*}{LB}}
& \multicolumn{1}{c}{\multirow[c]{2}{*}{Gap}}
& \multicolumn{1}{c}{Time}
& Root
& \multicolumn{1}{c}{Root}
& Cols.
& \#
& Br.
& \#
& \multicolumn{1}{c}{\#} \\
& & & & & &
& \multicolumn{1}{c}{(s)}
& UB
& \multicolumn{1}{c}{LB}
& W/S/F
& SR
& P/A
& Split
& \multicolumn{1}{c}{Nodes} \\
\midrule
\multirow[c]{2}{*}{S1}
& \multirow[c]{2}{*}{PS}
& \multirow[c]{2}{*}{31}
& MIQP & 0.263 & 0.263 & 0.00\% & \textbf{0.23} & - & 0.126 & - & - & - & - & \num{137} \\
& & & BPC & 0.263 & 0.263 & 0.00\% & 4.26 & 0.263 & \textbf{0.263} & 2/9/0 & 0 & 0/0 & 0 & \num{1} \\
\midrule
\multirow[c]{2}{*}{S2}
& \multirow[c]{2}{*}{PS}
& \multirow[c]{2}{*}{23}
& MIQP & 0.293 & 0.293 & 0.00\% & \textbf{0.13} & - & 0.171 & - & - & - & - & \num{33} \\
& & & BPC & 0.293 & 0.293 & 0.00\% & 4.91 & 0.293 & \textbf{0.293} & 2/9/0 & 0 & 0/0 & 0 & \num{1} \\
\midrule
\multirow[c]{2}{*}{S3}
& \multirow[c]{2}{*}{PS}
& \multirow[c]{2}{*}{23}
& MIQP & 0.258 & 0.258 & 0.00\% & \textbf{0.05} & - & 0.188 & - & - & - & - & \num{1} \\
& & & BPC & 0.258 & 0.258 & 0.00\% & 4.47 & 0.258 & \textbf{0.258} & 2/5/0 & 0 & 0/0 & 0 & \num{1} \\
\midrule
\multirow[c]{2}{*}{S4}
& \multirow[c]{2}{*}{PS}
& \multirow[c]{2}{*}{23}
& MIQP & 0.272 & 0.272 & 0.00\% & \textbf{0.12} & - & 0.271 & - & - & - & - & \num{35} \\
& & & BPC & 0.272 & 0.272 & 0.00\% & 4.42 & 0.272 & \textbf{0.272} & 2/8/0 & 0 & 0/0 & 0 & \num{1} \\
\midrule
\multirow[c]{2}{*}{S5}
& \multirow[c]{2}{*}{PC}
& \multirow[c]{2}{*}{31}
& MIQP & 0.007 & 0.007 & 0.00\% & \textbf{0.05} & - & 0.007 & - & - & - & - & \num{1} \\
& & & BPC & 0.007 & 0.007 & 0.00\% & 4.69 & 0.007 & 0.007 & 1/10/0 & 0 & 0/0 & 0 & \num{1} \\
\midrule
\multirow[c]{2}{*}{S6}
& \multirow[c]{2}{*}{PC}
& \multirow[c]{2}{*}{33}
& MIQP & 0.187 & 0.187 & 0.00\% & \textbf{0.09} & - & 0.187 & - & - & - & - & \num{1} \\
& & & BPC & 0.187 & 0.187 & 0.00\% & 4.34 & 0.187 & 0.187 & 2/36/0 & 3 & 0/0 & 0 & \num{1} \\
\midrule
\multirow[c]{2}{*}{S7}
& \multirow[c]{2}{*}{PC}
& \multirow[c]{2}{*}{35}
& MIQP & 0.005 & 0.005 & 0.00\% & \textbf{0.11} & - & 0.005 & - & - & - & - & \num{1} \\
& & & BPC & 0.005 & 0.005 & 0.00\% & 4.60 & 0.005 & 0.005 & 1/26/0 & 0 & 0/0 & 0 & \num{1} \\
\midrule
\multirow[c]{2}{*}{S8}
& \multirow[c]{2}{*}{PC}
& \multirow[c]{2}{*}{30}
& MIQP & 0.222 & 0.222 & 0.00\% & \textbf{0.25} & - & 0.222 & - & - & - & - & \num{50} \\
& & & BPC & 0.222 & 0.222 & 0.00\% & 5.97 & 0.222 & 0.222 & 2/19/0 & 0 & 0/0 & 0 & \num{1} \\
\midrule
\multirow[c]{2}{*}{Average}
& \multirow[c]{2}{*}{-}
& \multirow[c]{2}{*}{28.6}
& MIQP & 0.188 & 0.188 & 0.00\% & \textbf{0.13} & - & 0.147 & - & - & - & - & \num{32.4} \\
& & & BPC & 0.188 & 0.188 & 0.00\% & 4.71 & 0.188 & \textbf{0.188} & 1.8/15.3/0.0 & 0.4 & 0.0/0.0 & 0.0 & \num{1.0} \\
\bottomrule
\end{tabular*}
\end{table}

\begin{table}[htbp]
\centering
\scriptsize
\renewcommand{\arraystretch}{0.9}
\setlength{\tabcolsep}{2pt}
\caption{Solution Details of Medium-Scale Instances
$(|\CustNodes|,|\Trucks|,|\HubNodes|,\DroneFleetSize)=(20,4,3,4)$}
\label{tab:medium-scale-results}
\begin{tabular*}{\textwidth}{
@{\extracolsep{\fill}}
cccccrrrcrccccr
@{}}
\toprule
\multirow[c]{2}{*}{Inst.}
& \multirow[c]{2}{*}{Patt.}
& \multirow[c]{2}{*}{$|\Duplicate{\Arcs}|$}
& \multirow[c]{2}{*}{Method}
& \multirow[c]{2}{*}{Inc.}
& \multicolumn{1}{c}{\multirow[c]{2}{*}{LB}}
& \multicolumn{1}{c}{\multirow[c]{2}{*}{Gap}}
& \multicolumn{1}{c}{Time}
& Root
& \multicolumn{1}{c}{Root}
& Cols.
& \#
& Br.
& \#
& \multicolumn{1}{c}{\#} \\
& & & & & &
& \multicolumn{1}{c}{(s)}
& UB
& \multicolumn{1}{c}{LB}
& W/S/F
& SR
& P/A
& Split
& \multicolumn{1}{c}{Nodes} \\
\midrule
\multirow[c]{2}{*}{M1}
& \multirow[c]{2}{*}{PS}
& \multirow[c]{2}{*}{168}
& MIQP & 0.186 & 0.186 & 0.00\% & \textbf{63.30} & - & -0.036 & - & - & - & - & \num{23495} \\
& & & BPC & 0.186 & 0.186 & 0.00\% & 75.40 & 0.186 & \textbf{0.186} & 4/341/0 & 0 & 0/0 & 0 & \num{1} \\
\midrule
\multirow[c]{2}{*}{M2}
& \multirow[c]{2}{*}{PS}
& \multirow[c]{2}{*}{208}
& MIQP & 0.167 & 0.109 & 34.99\% & 1800.00 & - & -0.098 & - & - & - & - & \num{148340} \\
& & & BPC & 0.167 & \textbf{0.167} & \textbf{0.00\%} & \textbf{85.33} & 0.167 & \textbf{0.167} & 4/495/0 & 0 & 0/0 & 0 & \num{1} \\
\midrule
\multirow[c]{2}{*}{M3}
& \multirow[c]{2}{*}{PS}
& \multirow[c]{2}{*}{211}
& MIQP & 0.187 & 0.187 & 0.00\% & \textbf{27.52} & - & -0.087 & - & - & - & - & \num{39124} \\
& & & BPC & 0.187 & 0.187 & 0.00\% & 58.85 & 0.187 & \textbf{0.187} & 4/346/0 & 7 & 0/0 & 0 & \num{1} \\
\midrule
\multirow[c]{2}{*}{M4}
& \multirow[c]{2}{*}{PS}
& \multirow[c]{2}{*}{216}
& MIQP & 0.187 & 0.031 & 83.28\% & 1800.00 & - & -0.107 & - & - & - & - & \num{557907} \\
& & & BPC & \textbf{0.186} & \textbf{0.186} & \textbf{0.00\%} & \textbf{72.91} & 0.186 & \textbf{0.186} & 4/297/0 & 0 & 0/0 & 0 & \num{1} \\
\midrule
\multirow[c]{2}{*}{M5}
& \multirow[c]{2}{*}{PC}
& \multirow[c]{2}{*}{213}
& MIQP & 0.127 & -0.050 & 139.63\% & 1800.00 & - & -0.162 & - & - & - & - & \num{155867} \\
& & & BPC & \textbf{0.123} & \textbf{0.123} & \textbf{0.00\%} & \textbf{92.01} & 0.127 & \textbf{0.123} & 4/505/0 & 23 & 0/0 & 0 & \num{1} \\
\midrule
\multirow[c]{2}{*}{M6}
& \multirow[c]{2}{*}{PC}
& \multirow[c]{2}{*}{218}
& MIQP & 0.059 & -0.034 & 157.62\% & 1800.00 & - & -0.134 & - & - & - & - & \num{572882} \\
& & & BPC & 0.059 & \textbf{0.059} & \textbf{0.00\%} & \textbf{303.48} & 0.062 & \textbf{0.059} & 3/952/0 & 32 & 0/0 & 39 & \num{1} \\
\midrule
\multirow[c]{2}{*}{M7}
& \multirow[c]{2}{*}{PC}
& \multirow[c]{2}{*}{168}
& MIQP & 0.120 & 0.120 & 0.00\% & \textbf{287.44} & - & -0.093 & - & - & - & - & \num{160916} \\
& & & BPC & 0.120 & 0.120 & 0.00\% & 1414.80 & 0.121 & \textbf{0.106} & 4/2412/0 & 371 & 45/0 & 4 & \num{91} \\
\midrule
\multirow[c]{2}{*}{M8}
& \multirow[c]{2}{*}{PC}
& \multirow[c]{2}{*}{168}
& MIQP & 0.174 & 0.174 & 0.00\% & 117.75 & - & -0.079 & - & - & - & - & \num{183760} \\
& & & BPC & 0.174 & 0.174 & 0.00\% & \textbf{76.50} & 0.174 & \textbf{0.174} & 4/331/0 & 24 & 0/0 & 0 & \num{1} \\
\midrule
\multirow[c]{2}{*}{Average}
& \multirow[c]{2}{*}{-}
& \multirow[c]{2}{*}{196.3}
& MIQP & 0.151 & 0.090 & 51.94\% & 962.18 & - & -0.100 & - & - & - & - & \num{230286.4} \\
& & & BPC & \textbf{0.150} & \textbf{0.150} & \textbf{0.00\%} & \textbf{272.41} & 0.151 & \textbf{0.149} & 3.9/709.9/0.0 & 57.1 & 5.6/0.0 & 5.4 & \num{12.3} \\
\bottomrule
\end{tabular*}
\end{table}

\begin{table}[htbp]
\centering
\scriptsize
\renewcommand{\arraystretch}{0.9}
\setlength{\tabcolsep}{2pt}
\caption{Solution Details of Large-Scale Instances
$(|\CustNodes|,|\Trucks|,|\HubNodes|,\DroneFleetSize)=(30,6,3,4)$}
\label{tab:large-scale-results}
\begin{tabular*}{\textwidth}{
@{\extracolsep{\fill}}
cccccrrrcrccccr
@{}}
\toprule
\multirow[c]{2}{*}{Inst.}
& \multirow[c]{2}{*}{Patt.}
& \multirow[c]{2}{*}{$|\Duplicate{\Arcs}|$}
& \multirow[c]{2}{*}{Method}
& \multirow[c]{2}{*}{Inc.}
& \multicolumn{1}{c}{\multirow[c]{2}{*}{LB}}
& \multicolumn{1}{c}{\multirow[c]{2}{*}{Gap}}
& \multicolumn{1}{c}{Time}
& Root
& \multicolumn{1}{c}{Root}
& Cols.
& \#
& Br.
& \#
& \multicolumn{1}{c}{\#} \\
& & & & & &
& \multicolumn{1}{c}{(s)}
& UB
& \multicolumn{1}{c}{LB}
& W/S/F
& SR
& P/A
& Split
& \multicolumn{1}{c}{Nodes} \\
\midrule
\multirow[c]{2}{*}{L1}
& \multirow[c]{2}{*}{PS}
& \multirow[c]{2}{*}{371}
& MIQP & - & -0.059 & - & 1800.00 & - & -0.117 & - & - & - & - & \num{46742} \\
& & & BPC & \textbf{0.163} & \textbf{0.163} & \textbf{0.00\%} & \textbf{395.87} & - & \textbf{0.162} & 0/625/262 & 74 & 1/0 & 0 & \num{3} \\
\midrule
\multirow[c]{2}{*}{L2}
& \multirow[c]{2}{*}{PS}
& \multirow[c]{2}{*}{404}
& MIQP & 0.111 & -0.095 & 185.21\% & 1800.00 & - & -0.105 & - & - & - & - & \num{25650} \\
& & & BPC & 0.111 & \textbf{0.109} & \textbf{1.81\%} & 1800.00 & - & \textbf{0.106} & 0/1259/380 & 117 & 5/0 & 73 & \num{7} \\
\midrule
\multirow[c]{2}{*}{L3}
& \multirow[c]{2}{*}{PS}
& \multirow[c]{2}{*}{360}
& MIQP & 0.134 & -0.069 & 151.63\% & 1800.00 & - & -0.106 & - & - & - & - & \num{52155} \\
& & & BPC & \textbf{0.127} & \textbf{0.127} & \textbf{0.00\%} & \textbf{254.87} & 0.135 & \textbf{0.127} & 6/858/0 & 25 & 0/0 & 0 & \num{1} \\
\midrule
\multirow[c]{2}{*}{L4}
& \multirow[c]{2}{*}{PS}
& \multirow[c]{2}{*}{346}
& MIQP & 0.142 & -0.055 & 138.58\% & 1800.00 & - & -0.072 & - & - & - & - & \num{36401} \\
& & & BPC & \textbf{0.140} & \textbf{0.140} & \textbf{0.00\%} & \textbf{114.08} & 0.220 & \textbf{0.140} & 6/836/0 & 0 & 0/0 & 0 & \num{1} \\
\midrule
\multirow[c]{2}{*}{L5}
& \multirow[c]{2}{*}{PC}
& \multirow[c]{2}{*}{545}
& MIQP & 0.062 & -0.100 & 260.46\% & 1800.00 & - & -0.108 & - & - & - & - & \num{54534} \\
& & & BPC & 0.062 & \textbf{0.062} & \textbf{0.00\%} & \textbf{1086.79} & 0.080 & \textbf{0.062} & 6/1503/0 & 8 & 0/0 & 74 & \num{1} \\
\midrule
\multirow[c]{2}{*}{L6}
& \multirow[c]{2}{*}{PC}
& \multirow[c]{2}{*}{317}
& MIQP & 0.127 & -0.054 & 142.38\% & 1800.00 & - & -0.061 & - & - & - & - & \num{23290} \\
& & & BPC & \textbf{0.122} & \textbf{0.122} & \textbf{0.00\%} & \textbf{150.24} & 0.219 & \textbf{0.122} & 6/1295/0 & 0 & 0/0 & 0 & \num{1} \\
\midrule
\multirow[c]{2}{*}{L7}
& \multirow[c]{2}{*}{PC}
& \multirow[c]{2}{*}{317}
& MIQP & 0.167 & 0.026 & 84.21\% & 1800.00 & - & 0.019 & - & - & - & - & \num{46923} \\
& & & BPC & \textbf{0.158} & \textbf{0.158} & \textbf{0.00\%} & \textbf{128.73} & - & \textbf{0.158} & 0/440/275 & 0 & 0/0 & 0 & \num{1} \\
\midrule
\multirow[c]{2}{*}{L8}
& \multirow[c]{2}{*}{PC}
& \multirow[c]{2}{*}{438}
& MIQP & 0.100 & -0.096 & 196.45\% & 1800.00 & - & -0.098 & - & - & - & - & \num{27763} \\
& & & BPC & \textbf{0.088} & \textbf{0.085} & \textbf{3.58\%} & 1800.00 & 0.104 & \textbf{0.085} & 6/1270/0 & 107 & 2/0 & 237 & \num{3} \\
\midrule
\multirow[c]{2}{*}{Average}
& \multirow[c]{2}{*}{-}
& \multirow[c]{2}{*}{387.3}
& MIQP & - & -0.063 & - & 1800.00 & - & -0.081 & - & - & - & - & \num{39182.3} \\
& & & BPC & \textbf{0.121} & \textbf{0.121} & \textbf{0.67\%} & \textbf{716.65} & 0.152 & \textbf{0.120} & 3.8/1010.8/114.6 & 41.4 & 1.0/0.0 & 48.0 & \num{2.3} \\
\bottomrule
\end{tabular*}
\end{table}

Taken together, \autoref{tab:small-scale-results}--\autoref{tab:large-scale-results} show a clear change in the relative performance of the two solution approaches as the problem scale increases. For the small instances, BPC and the compact MIQP return identical incumbent solutions and global lower bounds, providing empirical validation of the exactness of the proposed BPC. The compact MIQP is nevertheless considerably faster because the additional machinery of column generation is unnecessary at this scale. At the medium scale, the two methods still obtain very similar incumbent solutions on average, but their ability to certify solution quality differs substantially: BPC proves optimality for all instances, whereas the average global lower bound of the MIQP remains much weaker. This difference becomes more pronounced at the large scale. In addition to maintaining substantially tighter global lower bounds, BPC obtains better incumbent solutions for most instances and proves optimality for six of the eight cases, while the compact MIQP reaches the time limit on all instances.

The stronger certification performance mainly comes from the tighter LP relaxation of the path-based set-partitioning formulation used by BPC. The average root LB improves from $-0.100$ to $0.149$ for the medium-scale instances and from $-0.081$ to $0.120$ for the large-scale instances. Accordingly, most BPC instances are solved at the root node or within a small branch-and-bound tree. Whenever branching is required, all recorded decisions use the customer-pair rule; the customer-conditioned transformed-arc fallback is not invoked in these experiments.

M7 is the main exception to this root-node behavior. Its root relaxation retains an 11.89\% gap and requires 45 customer-pair branching decisions, resulting in 91 processed nodes and 371 SR inequalities. However, the overall results do not isolate how the individual algorithmic components contribute to BPC performance. This motivates the following controlled component analysis of parallel pricing, dynamic task splitting, and subset-row inequalities.

\subsection{Controlled Component Analysis}
To examine how the main algorithmic components contribute to BPC performance, we conduct a controlled experiment on four diagnostic large-scale instances, L1, L2, L5, and L8. These instances cover both PS and PC layouts and different computational regimes. Three variants are considered. \emph{BPC-SingleProc} retains all BPC components but uses a single pricing process, providing a baseline for examining the benefit of parallel pricing. \emph{BPC-NoSplit} uses 12 pricing processes and the same score-balanced initial route-space partition as Full BPC, but disables subsequent task repartitioning, thereby examining the contribution of dynamic task splitting. \emph{BPC-NoSR} retains parallel pricing and dynamic splitting but disables subset-row inequalities, thereby examining the contribution of SR cuts. All other pricing, dominance, pruning, branching, and primal procedures remain unchanged. The corresponding MIQP results from \autoref{tab:large-scale-results} are included as an external reference.

\begin{table}[tbp]
\centering
\scriptsize
\renewcommand{\arraystretch}{0.9}
\setlength{\tabcolsep}{3pt}
\caption{Controlled Component Analysis on Selected Large-Scale Instances}
\label{tab:controlled-experiment}
\begin{tabular*}{\textwidth}{
@{\extracolsep{\fill}}
cccrrrrrrrr
@{}}
\toprule
\multirow[c]{2}{*}{Inst.}
& \multirow[c]{2}{*}{Patt.}
& \multirow[c]{2}{*}{Method}
& \multicolumn{1}{c}{Inc.}
& \multicolumn{1}{c}{LB}
& \multicolumn{1}{c}{Gap}
& \multicolumn{1}{c}{Time}
& \multicolumn{1}{c}{Root LB}
& \multicolumn{1}{c}{\# SR}
& \multicolumn{1}{c}{\# Split}
& \multicolumn{1}{c}{\# Nodes} \\
& & & & & & \multicolumn{1}{c}{(s)} & & & & \\
\midrule

\multirow[c]{5}{*}{L1}
& \multirow[c]{5}{*}{PS}
& MIQP         & -     & -0.059          & -                & 1800.00          & -0.117          & -  & - & \num{46742} \\
& & BPC-SingleProc & 0.163 & 0.162           & 0.57\%           & 1800.00          & \textbf{0.162}   & 71 & 0 & \num{3} \\
& & BPC-NoSplit & 0.163 & \textbf{0.163}  & \textbf{0.00\%}  & 351.17           & \textbf{0.162}   & 72 & 0 & \num{3} \\
& & BPC-NoSR   & 0.163 & \textbf{0.163}  & \textbf{0.00\%}  & \textbf{151.83}  & \textbf{0.162}   & 0  & 0 & \num{7} \\
& & Full BPC   & 0.163 & \textbf{0.163}  & \textbf{0.00\%}  & 395.87           & \textbf{0.162}   & 74 & 0 & \num{3} \\
\midrule

\multirow[c]{5}{*}{L2}
& \multirow[c]{5}{*}{PS}
& MIQP         & 0.111 & -0.095          & 185.21\%         & 1800.00          & -0.105          & -   & -  & \num{25650} \\
& & BPC-SingleProc & -     & -0.231          & -                & 1800.00          & -               & 0   & 0  & \num{1} \\
& & BPC-NoSplit & 0.112 & 0.106           & 5.08\%           & 1800.00          & \textbf{0.106}  & 98  & 0  & \num{3} \\
& & BPC-NoSR   & 0.110 & \textbf{0.110}  & \textbf{0.00\%}  & \textbf{1216.67} & 0.105           & 0   & 18 & \num{47} \\
& & Full BPC   & 0.111 & 0.109           & 1.81\%           & 1800.00          & \textbf{0.106}  & 117 & 73 & \num{7} \\
\midrule

\multirow[c]{5}{*}{L5}
& \multirow[c]{5}{*}{PC}
& MIQP         & 0.062 & -0.100          & 260.46\%         & 1800.00          & -0.108          & - & -  & \num{54534} \\
& & BPC-SingleProc & 0.075 & -0.234          & 412.33\%         & 1800.00          & -               & 0 & 0  & \num{1} \\
& & BPC-NoSplit & 0.086 & -0.234          & 373.29\%         & 1800.00          & -               & 8 & 0  & \num{1} \\
& & BPC-NoSR   & 0.062 & \textbf{0.062}  & \textbf{0.00\%}  & \textbf{816.23}  & \textbf{0.062}  & 0 & 65 & \num{3} \\
& & Full BPC   & 0.062 & \textbf{0.062}  & \textbf{0.00\%}  & 1086.79          & \textbf{0.062}  & 8 & 74 & \num{1} \\
\midrule

\multirow[c]{5}{*}{L8}
& \multirow[c]{5}{*}{PC}
& MIQP         & 0.100 & -0.096          & 196.45\%         & 1800.00 & -0.098          & -   & -   & \num{27763} \\
& & BPC-SingleProc & -     & -0.233          & -                & 1800.00 & -               & 0   & 0   & \num{1} \\
& & BPC-NoSplit & 0.177 & -0.233          & 231.48\%         & 1800.00 & -               & 59  & 0   & \num{1} \\
& & BPC-NoSR   & 0.092 & \textbf{0.085}  & 6.88\%           & 1800.00 & 0.082           & 0   & 100 & \num{19} \\
& & Full BPC   & 0.088 & \textbf{0.085}  & \textbf{3.58\%}  & 1800.00 & \textbf{0.085}  & 107 & 237 & \num{3} \\

\bottomrule
\end{tabular*}
\end{table}

The controlled results first examine the contribution of parallel pricing. BPC-SingleProc reaches the 1,800-s limit on all four instances and makes particularly limited progress on L2 and L8, where it does not obtain an incumbent. Moreover, root pricing does not close on L2, L5, and L8, so the reported lower bounds of $-0.231$, $-0.234$, and $-0.233$, respectively, are only constant coarse lower bounds determined by the objective-normalization constants. In comparison, BPC-NoSplit obtains finite incumbents on all four instances and completes root pricing on two instances, showing that the use of 12 pricing processes improves overall search progress. Parallel pricing without dynamic splitting (BPC-NoSplit), however, remains insufficient for the more pricing-intensive cases. It closes the root on L2 and terminates with a 5.08\% gap, but on L5 and L8 root pricing remains incomplete and the reported lower bounds are again the constant coarse bounds; hence, the corresponding gaps of 373.29\% and 231.48\% mainly reflect the absence of a certified root bound. With dynamic splitting, Full BPC solves L5 to optimality and reduces the gap on L8 from 231.48\% to 3.58\%, while also improving the L2 gap from 5.08\% to 1.81\%. In contrast, no splitting is triggered on L1, where BPC-NoSplit and Full BPC both solve the instance to optimality. These results show that parallel pricing provides the basic computational gain, while dynamic splitting is particularly important when difficult pricing subspaces prevent root pricing from closing.

The comparison between BPC-NoSR and Full BPC shows that SR inequalities reduce the branch-and-bound search. Full BPC processes 3.5 nodes on average, compared with 19.0 for BPC-NoSR, although this reduction comes at the cost of additional cut-separation and pricing effort. Consequently, BPC-NoSR is faster on L1 and L5 and even proves optimality on L2 within 1,216.67 s, whereas Full BPC reaches the time limit on L2 with a 1.81\% gap. On the more difficult L8 instance, however, Full BPC obtains a stronger root LB of 0.085 compared with 0.082 for BPC-NoSR and reduces the final gap from 6.88\% to 3.58\% with fewer explored nodes. These results indicate that SR inequalities are effective in reducing fractionality and controlling tree growth.

\subsection{Sensitivity Analysis and Managerial Insights}
\label{sec:sensitivity}
The sensitivity analysis examines two complementary aspects of the operating environment: the economic attractiveness of drone service and the managerial priorities embedded in the objective function. We first examine the effect of the per-sortie drone cost and then investigate how different weights on delivery delay, truck return time, and operating cost influence the resulting operating decisions. The experiments use four sensitivity-specific medium-scale instances, comprising two PS and two PC instances. In each experiment, only the setting under investigation is changed, while all other instance settings are kept at their baseline values. For each tested setting, the PS and PC results report the arithmetic averages over their respective two instances, whereas ``Overall'' reports the average over all four selected instances.
Across these experiments, four metrics are used to characterize the resulting operating plans. The drone service ratio is the percentage of customers served by drones. Root mean square (RMS) delay is defined as $\sqrt{\frac{1}{|\CustNodes|}\sum_{c\in\CustNodes}(\RealizedServTime{c}-\arriveLB{c})^2}$ and measures the overall severity of customer delays. Delay variation is defined as $\sqrt{\frac{1}{|\CustNodes|}\sum_{c\in\CustNodes}\bigl[(\RealizedServTime{c}-\arriveLB{c})-\bar d\bigr]^2}$, where $\bar d:=\frac{1}{|\CustNodes|}\sum_{c\in\CustNodes}(\RealizedServTime{c}-\arriveLB{c})$, and captures the dispersion of delays across customers. Operating cost includes the fixed cost of deployed trucks and the per-sortie cost of drone operations.

\begin{figure}[tbp]
  \centering
  \includegraphics[width=1\textwidth]{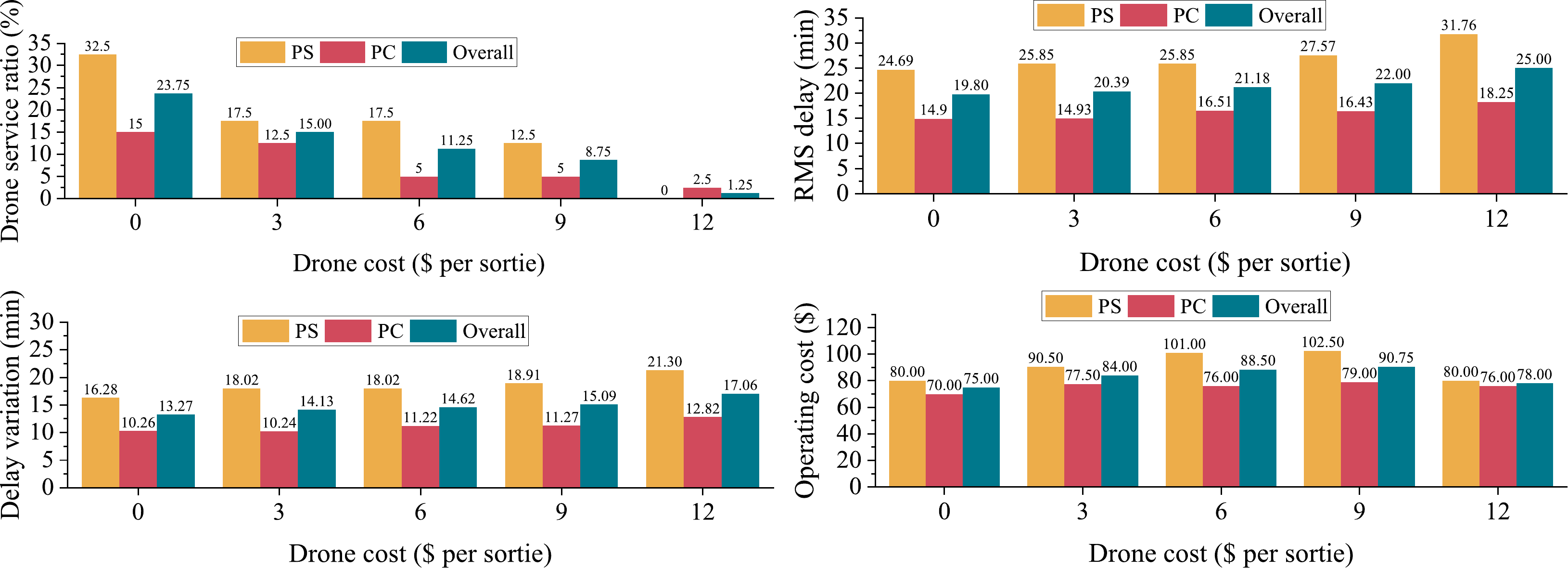}
  \caption{Effect of the drone cost on drone utilization, delivery-delay performance, and operating cost.}
  \label{fig:drone-cost-sensitivity}
\end{figure}

\autoref{fig:drone-cost-sensitivity} shows that increasing the per-sortie drone cost progressively shifts customer service from drones to trucks, with the overall drone service ratio decreasing by 94.7\% as the cost increases from \$0 to \$12. This substitution is accompanied by a 26.3\% increase in RMS delay and a 28.6\% increase in delay variation, indicating that reduced drone deployment not only increases customer delays overall but also makes service less uniform across customers. The deterioration is more pronounced for the PS instances, suggesting that drone service is particularly valuable when customers are geographically dispersed. The associated cost response is non-monotone: overall operating cost increases by 21.0\% as the sortie cost rises from \$0 to \$9, but then decreases by 14.0\% at \$12, when the drone service ratio falls to only 1.25\%. This pattern reveals a threshold beyond which further increases in drone cost reduce expenditure only by accepting poorer customer service performance.

\begin{figure}[tbp]
  \centering
  \includegraphics[width=1\textwidth]{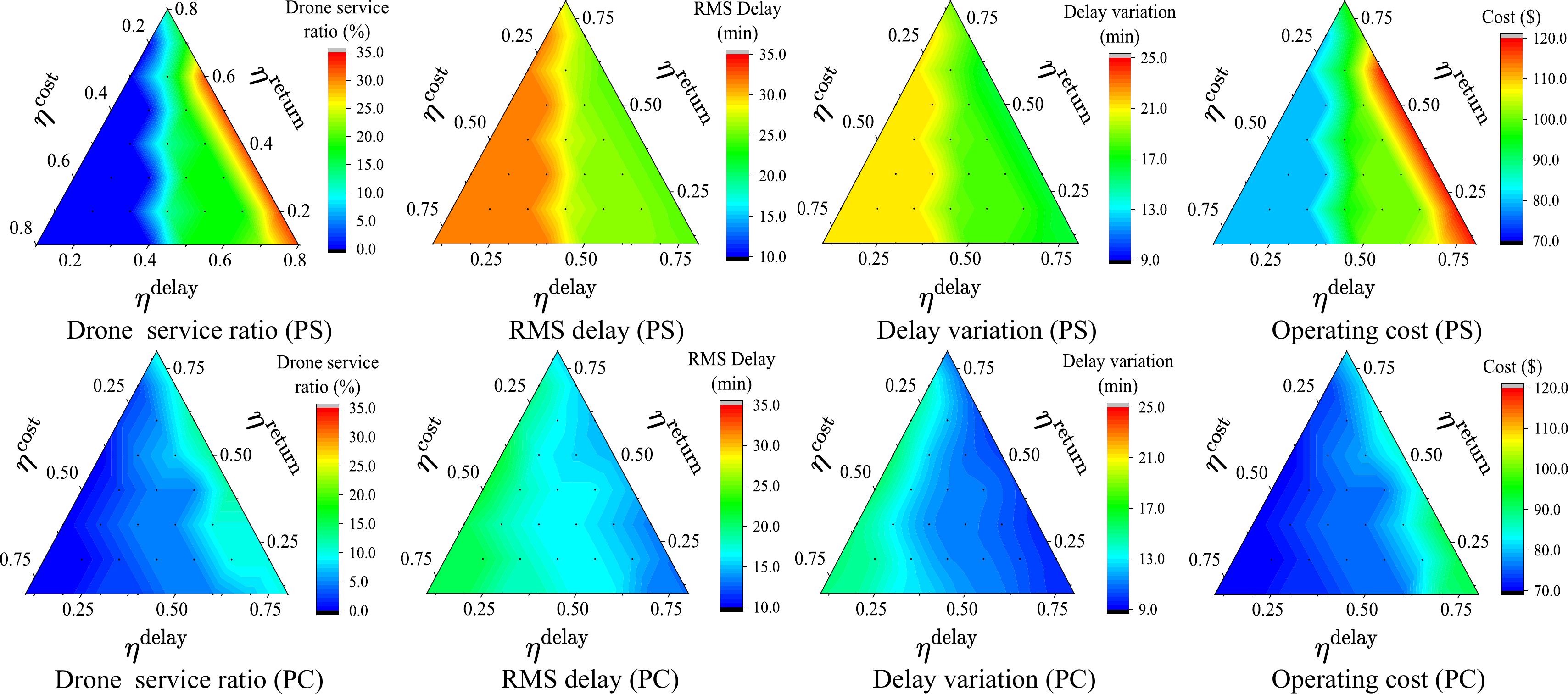}
  \caption{Operating trade-offs over the objective-weight simplex. Each ternary map associates a weight vector $(\WeightDelay,\WeightReturn,\WeightCost)$ with the indicated performance metric.}
  \label{fig:weight-sensitivity}
\end{figure}

\autoref{fig:weight-sensitivity} reveals that customer geography determines how managerial priorities are translated into operating decisions. For the PS instances, increasing the emphasis on delivery delay sharply increases drone utilization and operating cost while reducing both RMS delay and delay variation; greater emphasis on operating cost produces the opposite response. Thus, the service--cost trade-off in sparse networks is realized largely through mode substitution, with drones used to avoid costly truck service to geographically dispersed customers. The PC instances respond less strongly to changes in the objective weights. Increasing the emphasis on delivery delay still reduces RMS delay and delay variation, but it produces much smaller changes in drone utilization and operating cost than in PS. This is because clustered customers can often be served efficiently by trucks within relatively compact routes, so improving customer service does not require a large shift toward drone delivery. Notably, delay variation decreases together with RMS delay in both layouts, indicating that greater emphasis on the quadratic delay objective improves not only overall timeliness but also the uniformity of customer service.

The two time-related objectives induce different operating decisions. A larger $\WeightDelay$ prioritizes earlier customer service and may justify additional drone sorties to avoid late truck visits, whereas a larger $\WeightReturn$ prioritizes earlier route completion. They are not equivalent because a drone sortie can substantially advance a customer's service time without comparably reducing truck return time when detour savings are offset by synchronization-pad waiting. These results support treating customer delay separately from truck return time; the quadratic penalty also discourages large individual delays and service disparities.

\section{Conclusions}\label{sec:conc}

This paper introduces the Truck-as-Hub Vehicle Routing Problem with Drones in Rural Areas (\AbbrProb), in which trucks act as temporary mobile hubs at passive, predesignated synchronization pads. At each visited pad, a truck may dispatch multiple onboard drones simultaneously and must wait until all return before continuing its route. We formulate an arc-based mixed-integer quadratic program that jointly considers customer delivery delay, truck return time, and operating cost, with the quadratic delay term discouraging disproportionately poor service to remote customers.

To solve the problem exactly, we construct a canonical transformed network that provides a cost-preserving elementary-path representation of feasible truck--drone routes and an equivalent route-space set-partitioning formulation. Based on this representation, we develop an exact branch-price-and-cut (BPC) algorithm with tailored forward label-setting pricing, dominance and safe reduced-cost pruning, valid inequalities, and parallel pricing with dynamic route-space repartitioning.

Computational experiments show that BPC’s advantage grows with problem size. BPC proves optimality for all medium-scale and six of eight large-scale instances, with gaps of 1.81\% and 3.58\% on the remaining two, whereas the compact MIQP reaches the 1,800-second limit on all large instances. Sensitivity results show that higher drone costs shift service toward trucks and increase delays, particularly for spatially dispersed customers. They also show that customer delivery delay and truck return time induce distinct operating decisions, supporting explicit consideration of customer-level delay.

Several extensions merit further study. Jointly optimizing synchronization-pad locations and truck--drone routes would extend the framework toward integrated network design. Incorporating uncertain travel times and heterogeneous drone fleets would further improve the applicability of the \AbbrProb{} to realistic rural delivery operations.

\bibliographystyle{plainnat}
\bibliography{papers}

\clearpage
\appendix
\compactappendixparagraphs
\numberwithin{equation}{section}
\section{Forward Label-Setting Pricing}
\label{app:label-setting-details}

The forward label-setting algorithm constructs elementary source--sink paths in $\TransGraph$ while enforcing route resources and inherited branching decisions. A forward label ending at $i\in\Duplicate{\Nodes}$ is defined as $\ForwardLabel{i}=\ForwardLabelInfo{i}{}$.
It represents a partial path from $\DepotSource$ to $i$ and carries the following information. The component $\PartialPath{i}$ is the partial path in $\TransGraph$, given by the ordered sequence of transformed nodes from $\DepotSource$ to $i$, while $\VisitedTruckNodes{i}$ is the set of physical customer and synchronization-pad nodes already visited by the truck, together with the source depot. The component $\RealizedServTime{i}$ is the arrival time at node $i$. The resource $\TruckWeight{i}$ is the cumulative demand represented by the partial path, and $\LabelHub{i}$ is the most recently visited synchronization pad. The components $\TruckArrHubTime{i}$ and $\TruckWaitHubTime{i}$ are the truck arrival time at $\LabelHub{i}$ and the current waiting time at $\LabelHub{i}$ induced by the drone sorties in the active block, respectively, whereas $\HubFlyDroneNum{i}$ is the number of customers served by drones from $\LabelHub{i}$. The vector $\NumCustPiCoverByLabelSet{i}=\{\NumCustPiCoverByLabel{i}{\pi}:\pi\in\SRset\}$ contains the active-SR counters, where $\NumCustPiCoverByLabel{i}{\pi}$ is the number of customers from $\pi$ represented by the partial path. Finally, $\BranchState{i}$ records the branch-relevant status of the current partial path, and $\LabelReducedCost{i}$ is the reduced cost accumulated along the partial path.

For $j\in\CustNodes\cup\CustNodesDuplicate$, $\ServedCust{j}$ denotes the original customer represented by transformed node $j$: $\ServedCust{j}=j$ for $j\in\CustNodes$ and $\ServedCust{j}=\Origin{j}$ for $j\in\CustNodesDuplicate$. Thus, $\ServedCust{j}$ allows the same demand, service-envelope, elementarity, and dual calculations to be used for truck- and drone-served representations of a customer.
For any partial path $P$ in $\TransGraph$, define $\RouteServedSet{P}:=\{\ServedCust{j}\mid j\in P\cap(\CustNodes\cup\CustNodesDuplicate)\}$, which is the set of original customers represented by $P$.
For initialization, at $\DepotSource$, we set $\PartialPath{\DepotSource}=\{\DepotSource\}$ and $\VisitedTruckNodes{\DepotSource}=\{\DepotSource\}$. All time, payload, drone-count, and SR-counter resources are initialized to zero; each branch state is initialized to indicate that no branch-relevant event has yet occurred; $\LabelHub{\DepotSource}=-1$, and $\LabelReducedCost{\DepotSource}=\ObjCoeffCost\TruckCost-\kappa$.

\subsection{Feasible extensions.}\label{app:feasible-extensions}
A label $\ForwardLabel{i}$ may be extended to $j$ along $(i,j)\in\Duplicate{\Arcs}$ only if the resulting partial path satisfies the following conditions:

\paragraph{Depot and elementarity checks.} No extension is made when $i=\DepotSink$ or $j=\DepotSource$. For $j\in\CustNodes\cup\HubNodes$, physical-node elementarity requires $j\notin\VisitedTruckNodes{i}$. For $j\in\CustNodes\cup\CustNodesDuplicate$, customer elementarity requires $\ServedCust{j}\notin\RouteServedSet{\PartialPath{i}}$. A sink label is retained only if $\RouteServedSet{\PartialPath{\DepotSink}}\neq\emptyset$.

\paragraph{Active-pad and canonical-order checks.}
For arc $(i,j)\in\CustDupCustDupArcs$, customer $\ServedCust{j}$ must be accessible from the active pad, i.e., $\ServedCust{j}\in\DroneAccessNodes{\LabelHub{i}}$, and the extension must respect the canonical order, $\ServedCust{i}\prec_{\LabelHub{i}}\ServedCust{j}$. For arc $(i,j)\in\CustDupCustArcs$, we require $j\in\TruckAccessNodes{\LabelHub{i}}$.

\paragraph{Capacity checks.}
For $j\in\CustNodes\cup\CustNodesDuplicate$, payload feasibility requires $\TruckWeight{i}+\DemandWeight{\ServedCust{j}}\le\TruckPayload$. If $j\in\CustNodesDuplicate$, the extension additionally requires $\HubFlyDroneNum{i}+1\le\DroneFleetSize$. Drone payload and endurance feasibility are already enforced by the preprocessing of $\DroneArcs$.

\paragraph{Service-time checks.}
The arrival time at $j$ is updated according to the type of extension. For a truck arc $(i,j)\in\TruckArcs$, we set $\RealizedServTime{j}=\RealizedServTime{i}+\TruckTravelTime{i}{j}$. If $j\in\CustNodesDuplicate$, then $\RealizedServTime{j}=\TruckArrHubTime{i}+\DroneTravelTime{\LabelHub{i}}{\ServedCust{j}}$. If arc $(i,j)\in\CustDupCustArcs$, we set $\RealizedServTime{j}=\TruckArrHubTime{i}+\TruckWaitHubTime{i}+\TruckTravelTime{\LabelHub{i}}{j}$. Finally, whenever $j\in\CustNodes\cup\CustNodesDuplicate$, the extension is feasible only if $\RealizedServTime{j}\le\EffServUB{\ServedCust{j}}$.

\paragraph{Branching checks.}
The decisions in Section~\ref{sec:branching} are enforced during label extension. Under $\BranchSeparate{pq}$, an extension is rejected if it would cover both $p$ and $q$; under $\BranchTogether{pq}$, a label reaching $\DepotSink$ is retained only if $p\in\RouteServedSet{\PartialPath{\DepotSink}}\Longleftrightarrow q\in\RouteServedSet{\PartialPath{\DepotSink}}$. Under $\BranchArcDown{p}{u}{v}$, a label is rejected if its path both covers $p$ and traverses $(u,v)$; under $\BranchArcUp{p}{u}{v}$, a label reaching $\DepotSink$ covering $p$ is retained only if it traverses $(u,v)$.
The vector $\BranchState{i}$ stores unresolved branch history: for $\BranchTogether{pq}$ it records whether neither, only $p$, only $q$, or both are covered; for transformed-arc branches it records whether $p$ is covered and $(u,v)$ is traversed. $\BranchSeparate{pq}$ is checked locally and requires no stored state.

\paragraph{Resource extension.}
For a feasible extension, $\PartialPath{j}=\PartialPath{i}\oplus\{j\}$. If $j\in\CustNodes\cup\HubNodes$, set $\VisitedTruckNodes{j}:=\VisitedTruckNodes{i}\cup\{j\}$; otherwise, this resource is unchanged. If $j\in\CustNodes\cup\CustNodesDuplicate$, set $\TruckWeight{j}:=\TruckWeight{i}+\DemandWeight{\ServedCust{j}}$; otherwise, the payload is unchanged. If $j\in\HubNodes$, set $\LabelHub{j}:=j$, $\TruckArrHubTime{j}:=\RealizedServTime{j}$, and reset $\TruckWaitHubTime{j}=\HubFlyDroneNum{j}=0$. If $j\in\CustNodesDuplicate$, retain the active pad and its truck-arrival time, and set $\TruckWaitHubTime{j}:=\max\{\TruckWaitHubTime{i},\DroneTripTime{\LabelHub{i}}{\ServedCust{j}}\}$ and $\HubFlyDroneNum{j}:=\HubFlyDroneNum{i}+1$. At an original non-pad node, the active pad and its arrival time are retained, while the waiting time and drone count are reset to zero.
For each $\pi\in\SRset$, the SR counter is
$
\NumCustPiCoverByLabel{j}{\pi}
=
\NumCustPiCoverByLabel{i}{\pi}
+
\mathbf 1
\left\{
j\in\CustNodes\cup\CustNodesDuplicate,\,
\ServedCust{j}\in\pi
\right\}.
$
The set
$
\SRset(\NumCustPiCoverByLabelSet{i},\NumCustPiCoverByLabelSet{j})
:=
\left\{
\pi\in\SRset:
\NumCustPiCoverByLabel{i}{\pi}<2,\,
\NumCustPiCoverByLabel{j}{\pi}=2
\right\}
$
contains the active SR rows whose coefficient is first incurred by the extension. Then, the reduced-cost is updated as
\begin{equation}\label{eq:label-rc-extension}
\LabelReducedCost{j}
=
\LabelReducedCost{i}
-
\sum_{\pi\in
\SRset(\NumCustPiCoverByLabelSet{i},\NumCustPiCoverByLabelSet{j})}
\nu_\pi
+
\begin{cases}
\ObjCoeffCost\DroneCost
+\ObjCoeffDelay(\RealizedServTime{j}-\arriveLB{\ServedCust{j}})^2
-\mu_{\ServedCust{j}},
& j\in\CustNodesDuplicate,\\
\ObjCoeffDelay(\RealizedServTime{j}-\arriveLB{j})^2-\mu_j,
& j\in\CustNodes,\\
\ObjCoeffReturn\RealizedServTime{j},
& j=\DepotSink,\\
0,
& j\in\HubNodes.
\end{cases}
\end{equation}

\subsubsection{Dominance.}
For a label $\ForwardLabel{}$, let $\BranchExtSet{\ForwardLabel{}}$ be its branch-feasible remaining extensions and $\CompletionSet{\ForwardLabel{}}\subseteq\BranchExtSet{\ForwardLabel{}}$ those producing a nonempty, canonical, elementary, and resource-feasible source--sink path. Label $\ForwardLabel{i}^{1}$ dominates $\ForwardLabel{i}^{2}$ if $\CompletionSet{\ForwardLabel{i}^{2}}\subseteq\CompletionSet{\ForwardLabel{i}^{1}}$ and $\CompletionCost{\ForwardLabel{i}^{1}}{\Completion}\le\CompletionCost{\ForwardLabel{i}^{2}}{\Completion}$ for every $\Completion\in\CompletionSet{\ForwardLabel{i}^{2}}$. Branch compatibility requires $\BranchExtSet{\ForwardLabel{i}^{2}}\subseteq\BranchExtSet{\ForwardLabel{i}^{1}}$. To apply this definition efficiently during label setting, we next derive sufficient dominance rules in terms of the labels' stored resources, separately for labels at the same transformed node and for labels at different transformed nodes representing the same physical truck location.

\paragraph{Same-node dominance.}
Two labels are compared by the same-node rule only when they end at the same transformed node.
For any two labels $\ForwardLabel{u}^{1}$ and $\ForwardLabel{v}^{2}$, define $\SRCorrectionSet{\ForwardLabel{u}^{1}}{\ForwardLabel{v}^{2}}:=\{\pi\in\SRset:\nu_\pi<0,\,\NumCustPiCoverByLabel{u}{\pi}^{1}=1,\,\NumCustPiCoverByLabel{v}{\pi}^{2}=2\}$. This set contains the active SR rows for which the first label currently covers one customer of $\pi$ and the second covers two.

\begin{proposition}[Same-node forward dominance]
\label{improved_forward_dominance}
Label $\ForwardLabel{i}^{1}$ dominates $\ForwardLabel{i}^{2}$ if
\begin{equation}
\begin{aligned}
&\LabelHub{i}^{1}=\LabelHub{i}^{2},\qquad
\RouteServedSet{\PartialPath{i}^{1}}
\subseteq\RouteServedSet{\PartialPath{i}^{2}},\qquad
\VisitedTruckNodes{i}^1
\subseteq\VisitedTruckNodes{i}^{2},\qquad
\RouteServedSet{\PartialPath{i}^{1}}=\emptyset
\Longrightarrow
\RouteServedSet{\PartialPath{i}^{2}}=\emptyset.\\
&\TruckWeight{i}^{1}\le\TruckWeight{i}^{2},\qquad
\HubFlyDroneNum{i}^{1}\le\HubFlyDroneNum{i}^{2},\qquad
\BranchState{i}^{1}=\BranchState{i}^{2},\qquad
\RealizedServTime{i}^{1}\le\RealizedServTime{i}^{2},\\
&\TruckArrHubTime{i}^{1}\le\TruckArrHubTime{i}^{2},\qquad
\TruckWaitHubTime{i}^{1}\le\TruckWaitHubTime{i}^{2},\qquad
\LabelReducedCost{i}^{1}
-\sum_{\pi\in\SRCorrectionSet{\ForwardLabel{i}^{1}}{\ForwardLabel{i}^{2}}}\nu_\pi
\le\LabelReducedCost{i}^{2}.
\end{aligned}
\end{equation}
\end{proposition}
\begin{proof}
See Appendix~\ref{app:proof:improved-forward-dominance}. 
\end{proof}

If the two labels are equivalent under all dominance conditions, only one is retained, since either can replace the other.

\paragraph{Physical-location dominance.}
The dominance rule can also compare labels that end at different transformed nodes but correspond to the same physical truck location. For a label $\ForwardLabel{i}$, define $\PhysLoc{\ForwardLabel{i}}:=i$ for $i\in\Nodes$ and $\PhysLoc{\ForwardLabel{i}}:=\LabelHub{i}$ for $i\in\CustNodesDuplicate$. Thus, $\PhysLoc{\ForwardLabel{i}}$ is the latest physical truck location along the partial path.

For each $h\in\HubNodes$, arrange the customers in $\DroneAccessNodes{h}$ according to the canonical order $\prec_h$, and let $\operatorname{rank}_h(n)$ denote the position of customer $n$ in this ordered list. Define $\BlockPos{\ForwardLabel{i}}:=0$ for $i\in\HubNodes$ and $\BlockPos{\ForwardLabel{i}}:=\operatorname{rank}_{\LabelHub{i}}(\ServedCust{i})$ for $i\in\CustNodesDuplicate$. Thus, $\BlockPos{\ForwardLabel{i}}$ records the progress of the label through the canonical duplicated-customer sequence associated with its active pad.

Labels $\ForwardLabel{u}^{1}$ and $\ForwardLabel{v}^{2}$ are \emph{location-comparable} if either $u=v\in\Nodes\setminus\bigl(\HubNodes\cup\{\DepotSink\}\bigr)$, or $\PhysLoc{\ForwardLabel{u}^{1}}=\PhysLoc{\ForwardLabel{v}^{2}}\in\HubNodes$ and $\BlockPos{\ForwardLabel{u}^{1}}\le\BlockPos{\ForwardLabel{v}^{2}}$. In the latter case, $\ForwardLabel{u}^{1}$ is no farther along the canonical sequence than $\ForwardLabel{v}^{2}$. Hence, any duplicated customer satisfying the canonical-order condition from $\ForwardLabel{v}^{2}$ also satisfies it from $\ForwardLabel{u}^{1}$.

For location-comparable labels, define $\RetCredit{\ForwardLabel{u}^{1}}{\ForwardLabel{v}^{2}}:=\RealizedServTime{v}^{2}-\RealizedServTime{u}^{1}$ for $u=v\in\Nodes\setminus\bigl(\HubNodes\cup\{\DepotSink\}\bigr)$, and $\RetCredit{\ForwardLabel{u}^{1}}{\ForwardLabel{v}^{2}}:=\TruckArrHubTime{v}^{2}-\TruckArrHubTime{u}^{1}$ for $\PhysLoc{\ForwardLabel{u}^{1}}=\PhysLoc{\ForwardLabel{v}^{2}}\in\HubNodes$. This quantity is the part of the timing advantage of $\ForwardLabel{u}^{1}$ that is guaranteed to remain in the final truck return time. At a common synchronization pad, only the truck-arrival-time difference is guaranteed, because later drone extensions may make the two waiting times equal. The sink is excluded because its return-time contribution has already been included in the label reduced cost; sink labels are handled by Proposition~\ref{improved_forward_dominance}.

\begin{proposition}[Physical-location forward dominance]
\label{prop:physical-location-dominance}
Consider two location-comparable labels $\ForwardLabel{u}^{1}$ and $\ForwardLabel{v}^{2}$. When $u\neq v$, assume that, for every node $x$ that can be the first node of a feasible completion from $\ForwardLabel{v}^{2}$, the arc $(u,x)$ belongs to $\Duplicate{\Arcs}$ and the extensions along $(u,x)$ and $(v,x)$ yield the same branch-state vector. Then $\ForwardLabel{u}^{1}$ dominates $\ForwardLabel{v}^{2}$ if
\begin{align}
&\RouteServedSet{\PartialPath{u}^{1}}\subseteq\RouteServedSet{\PartialPath{v}^{2}},\qquad
\VisitedTruckNodes{u}^{1}\subseteq\VisitedTruckNodes{v}^{2},\qquad
\RouteServedSet{\PartialPath{u}^{1}}=\emptyset
\Longrightarrow
\RouteServedSet{\PartialPath{v}^{2}}=\emptyset,
\label{eq:pld-cust}\\
&\TruckWeight{u}^{1}\le\TruckWeight{v}^{2},\qquad
\HubFlyDroneNum{u}^{1}\le\HubFlyDroneNum{v}^{2},\qquad
\BranchState{u}^{1}=\BranchState{v}^{2},
\label{eq:pld-resource}\\
&u=v\in\Nodes\setminus\bigl(\HubNodes\cup\{\DepotSink\}\bigr)
\Longrightarrow
\RealizedServTime{u}^{1}\le\RealizedServTime{v}^{2},\\
&\PhysLoc{\ForwardLabel{u}^{1}}=\PhysLoc{\ForwardLabel{v}^{2}}\in\HubNodes
\Longrightarrow
\left(
\TruckArrHubTime{u}^{1}\le\TruckArrHubTime{v}^{2}
\;\wedge\;
\TruckWaitHubTime{u}^{1}\le\TruckWaitHubTime{v}^{2}
\right),
\label{eq:pld-time}\\
&\LabelReducedCost{u}^{1}
-\sum_{\pi\in\SRCorrectionSet{\ForwardLabel{u}^{1}}{\ForwardLabel{v}^{2}}}\nu_\pi
-\ObjCoeffReturn\RetCredit{\ForwardLabel{u}^{1}}{\ForwardLabel{v}^{2}}
\le
\LabelReducedCost{v}^{2}.
\label{eq:pld-cost}
\end{align}
\end{proposition}
\begin{proof}
See Appendix~\ref{app:proof:prop-physical-location-dominance}. 
\end{proof}

\section{Farkas Pricing}
\label{app:farkas-details}
At some branch-and-bound nodes, the current RMP may be infeasible because its column set is insufficient. We therefore use Farkas pricing to generate a column that violates the current Farkas certificate. Let $(\tilde\mu,\tilde\kappa,\tilde\nu)$ be a Farkas dual ray associated, respectively, with the customer-covering, truck-fleet, and active SR constraints of the RMP. The ray satisfies
$
\sum_{n\in\CustNodes} \tilde\mu_n\RouteVisitNode{r}{n}
+\tilde\kappa
+\sum_{\SRsubset\in\ActiveSRset{m}} \tilde\nu_{\SRsubset}\sigma_{r,\SRsubset}
\le 0,
$
for all $r\in\RouteSubset{m}$ and
$
\sum_{n\in\CustNodes}\tilde\mu_n
+|\Trucks|\tilde\kappa
+\sum_{\SRsubset\in\ActiveSRset{m}}\tilde\nu_{\SRsubset}
>0
$.
The ray components satisfy
$
\tilde\mu_n\in\mathbb{R},~
\tilde\kappa\le0,~
\tilde\nu_{\SRsubset}\le0,~
n\in\CustNodes,~
\SRsubset\in\ActiveSRset{m}.
$
Farkas pricing searches a route $r\in\BranchFeasibleRoutes{m}$ that violates the ray, i.e.,
$
\sum_{n\in\CustNodes} \tilde\mu_n\RouteVisitNode{r}{n}
+\sum_{\SRsubset\in\ActiveSRset{m}} \tilde\nu_{\SRsubset}\sigma_{r,\SRsubset}
+\tilde\kappa
>0
$.

Farkas pricing uses the same feasible-extension checks and the same initialization and update rules for all label components other than the reduced cost as standard pricing. The reduced-cost is initialized by $\LabelReducedCost{\DepotSource}=-\tilde\kappa$, and the update in \eqref{eq:label-rc-extension} is replaced by
\begin{equation}\label{eq:farkas-label-rc}
\LabelReducedCost{j} = \LabelReducedCost{i} - \sum_{\SRsubset\in\SRset(\NumCustPiCoverByLabelSet{i},\NumCustPiCoverByLabelSet{j})} \tilde\nu_{\SRsubset} - \begin{cases}
\tilde\mu_{\ServedCust{j}}, & j\in\CustNodes\cup\CustNodesDuplicate,\\
0, & \textnormal{otherwise}.
\end{cases}
\end{equation}

For a complete label corresponding to route $r$, recursion \eqref{eq:farkas-label-rc} gives $\LabelReducedCost{\DepotSink}=-\tilde\kappa-\sum_{n\in\CustNodes}\tilde\mu_n\RouteVisitNode{r}{n}-\sum_{\SRsubset\in\ActiveSRset{m}}\tilde\nu_{\SRsubset}\sigma_{r,\SRsubset}$. Hence, $\LabelReducedCost{\DepotSink}<0$ when $r$ satisfies the Farkas-violation condition stated above, in which case $r$ is added to the RMP. Farkas pricing uses the dominance rule below and the completion bound in Proposition~\ref{prop:farkas-dual-reward-pruning}, developed in Appendix~\ref{app:accel-full}. The standard-pricing completion bound does not apply directly because the Farkas reduced cost contains no route-cost terms.

\begin{proposition}[Farkas-pricing dominance]
\label{prop:farkas-pricing-dominance}
Consider two location-comparable labels $\ForwardLabel{u}^{1}$ and $\ForwardLabel{v}^{2}$. When $u\neq v$, require that, for every node $x$ that can be the first node of a feasible completion from $\ForwardLabel{v}^{2}$, the arc $(u,x)$ belongs to $\Duplicate{\Arcs}$ and the extensions along $(u,x)$ and $(v,x)$ yield the same branch-state vector at $x$. Suppose that conditions \eqref{eq:pld-cust}--\eqref{eq:pld-time} hold with $\NumCustPiCoverByLabelSet{u}^{1}=\NumCustPiCoverByLabelSet{v}^{2}$ and $\LabelReducedCost{u}^{1}\le\LabelReducedCost{v}^{2}$. Then $\ForwardLabel{u}^{1}$ dominates $\ForwardLabel{v}^{2}$ in Farkas pricing.
\end{proposition}
\begin{proof}
See Appendix~\ref{app:proof:prop-farkas-pricing-dominance}. 
\end{proof}

\section{Implementation Details and Primal Heuristics}
\label{app:accel-full}

\subsection{Customer-Specific Earliest-Service Lower Bounds}
\label{sec:arr_bound}
The quadratic delay term uses the customer-specific service-time lower bound $\arriveLB{n}$. Using the shortest truck travel times $\ShortestTruckTime{i}{j}$ defined in Section~\ref{arc-form}, direct truck service gives $\TruckServLB{n}:=\ShortestTruckTime{\DepotSource}{n}$ when $\DemandWeight{n}\le\TruckPayload$, $\ShortestTruckTime{\DepotSource}{n}<+\infty$, and $\ShortestTruckTime{n}{\DepotSink}<+\infty$; otherwise, $\TruckServLB{n}:=+\infty$.
For drone service, define $\DroneFeasibleHubs{n}:=\{h\in\HubNodes:(h,n)\in\DroneArcs,\,\ShortestTruckTime{\DepotSource}{h}<+\infty,\,\ShortestTruckTime{h}{\DepotSink}<+\infty\}$. If customer $n$ is served by a drone launched from $h\in\DroneFeasibleHubs{n}$, the truck cannot reach $h$ earlier than $\ShortestTruckTime{\DepotSource}{h}$, so the delivery time of $n$ is at least $\ShortestTruckTime{\DepotSource}{h}+\DroneTravelTime{h}{n}$. Hence,
$\DroneServLB{n}:=\min_{h\in\DroneFeasibleHubs{n}}\{\ShortestTruckTime{\DepotSource}{h}+\DroneTravelTime{h}{n}\}$, with the minimum defined as $+\infty$ when $\DroneFeasibleHubs{n}=\emptyset$.
Since every customer is served either by truck or by drone, a valid lower bound on its service time is $\arriveLB{n}:=\min\{\TruckServLB{n},\DroneServLB{n}\}$ for all $n\in\CustNodes$.

\subsection{Reachability and Dual-Reward Pruning.}
\label{sec:dual-reward-pruning}
For a label $\ForwardLabel{i}$, define its residual payload as $\ResidualPayload{\ForwardLabel{i}}:=\TruckPayload-\TruckWeight{i}$. Also, every feasible completion has truck return time at least $\CompletionReturnLB{\ForwardLabel{i}}:=\RealizedServTime{i}+\ShortestTruckTime{i}{\DepotSink}$ if $i\in\Nodes$, and $\CompletionReturnLB{\ForwardLabel{i}}:=\TruckArrHubTime{i}+\TruckWaitHubTime{i}+\ShortestTruckTime{\LabelHub{i}}{\DepotSink}$ if $i\in\CustNodesDuplicate$. These two quantities are used below to bound feasible continuations and their reduced costs.
For $v\in\Nodes$ and $n\in\CustNodes$, define the relaxed earliest service time from physical location $v$ as
\[
\beta_{v,n} := \min\left\{\ShortestTruckTime{v}{n},\,\min_{h:(h,n)\in\DroneArcs}\left[\ShortestTruckTime{v}{h}+\DroneTravelTime{h}{n}\right]\right\},
\]
where an empty minimum is interpreted as $+\infty$. Thus, $\beta_{v,n}$ considers either direct truck service of $n$ or future drone service from a pad reachable from $v$. Also define $\zeta_{h,n}:=\DroneTravelTime{h}{n}$ if $(h,n)\in\DroneArcs$, and $\zeta_{h,n}:=+\infty$ otherwise.

For $n\in\CustNodes\setminus\RouteServedSet{\PartialPath{i}}$, a lower bound on its service time over all feasible completions of $\ForwardLabel{i}$ is
\begin{equation}\label{eq:label-customer-service-lb}
\LabelServLB{\ForwardLabel{i}}{n} := \begin{cases}
\RealizedServTime{i}+\beta_{i,n}, & i\in\Nodes,\\[1mm]
\TruckArrHubTime{i}+\min\left\{\TruckWaitHubTime{i}+\beta_{\LabelHub{i},n},\,\zeta_{\LabelHub{i},n}\right\}, & i\in\CustNodesDuplicate.
\end{cases}
\end{equation}

At a physical endpoint, $\beta_{i,n}$ covers direct truck service and drone service from a future reachable pad. Within an active drone-service block, the second case also allows $n$ to be served immediately from the active pad before the block is closed. Hence, $\LabelServLB{\ForwardLabel{i}}{n}$ is valid for every feasible completion of $\ForwardLabel{i}$.

We then define
\begin{equation}\label{eq:reachable-customer-set}
\begin{aligned}
\ReachableCustSet{\ForwardLabel{i}}
:=
\Bigl\{
n\in\CustNodes\setminus\RouteServedSet{\PartialPath{i}}:\;
&\DemandWeight{n}\le\ResidualPayload{\ForwardLabel{i}},\;
\LabelServLB{\ForwardLabel{i}}{n}\le\EffServUB{n},\\
&
\bigl(
\ShortestTruckTime{\PhysLoc{\ForwardLabel{i}}}{n}<+\infty
\wedge
\ShortestTruckTime{n}{\DepotSink}<+\infty
\bigr)\\
&\quad\vee
\bigl(
\exists h\in\HubNodes:
(h,n)\in\DroneArcs,\;
\ShortestTruckTime{\PhysLoc{\ForwardLabel{i}}}{h}<+\infty,\;
\ShortestTruckTime{h}{\DepotSink}<+\infty
\bigr)
\Bigr\}.
\end{aligned}
\end{equation}

If a feasible completion of $\ForwardLabel{i}$ covers a customer not already in $\RouteServedSet{\PartialPath{i}}$, that customer belongs to $\ReachableCustSet{\ForwardLabel{i}}$. Consequently, an unserved customer outside $\ReachableCustSet{\ForwardLabel{i}}$ cannot be served in any feasible continuation of $\ForwardLabel{i}$.

\paragraph{Together-branch pruning.}
Under $\BranchTogether{pq}$, if $p\in\RouteServedSet{\PartialPath{i}}$, $q\notin\RouteServedSet{\PartialPath{i}}$, and $q\notin\ReachableCustSet{\ForwardLabel{i}}$, then $\ForwardLabel{i}$ has no branch-feasible completion and is discarded. The same rule applies with $p$ and $q$ exchanged. Since this rule depends only on feasibility, it applies to both standard and Farkas pricing.

\paragraph{Dual-reward reduced-cost pruning.}
For standard pricing, only positive customer duals can decrease the reduced cost of a completion. Since every newly covered customer belongs to $\ReachableCustSet{\ForwardLabel{i}}$ and the total additional demand cannot exceed $\ResidualPayload{\ForwardLabel{i}}$, define $\PositivePart{\mu_n}:=\max\{\mu_n,0\}$ and bound the attainable customer dual reward by the fractional knapsack 
\[
\DualRewardBound{\ForwardLabel{i}}:=\max\left\{
\begin{aligned}
&\sum_{n\in\ReachableCustSet{\ForwardLabel{i}}}\PositivePart{\mu_n}\KnapVar{n}:\sum_{n\in\ReachableCustSet{\ForwardLabel{i}}}\DemandWeight{n}\KnapVar{n}\le\ResidualPayload{\ForwardLabel{i}},\\
&0\le\KnapVar{n}\le1,\ n\in\ReachableCustSet{\ForwardLabel{i}}
\end{aligned}
\right\}.
\]
 It is solved greedily by ordering customers in nonincreasing $\PositivePart{\mu_n}/\DemandWeight{n}$, taking each customer fully while capacity remains and, if necessary, taking the first nonfitting customer fractionally to exhaust the residual payload. A lower bound on the reduced cost of every feasible completion is $\KnapReducedCostLB{\ForwardLabel{i}}:=\LabelReducedCost{i}+\ObjCoeffReturn\CompletionReturnLB{\ForwardLabel{i}}-\DualRewardBound{\ForwardLabel{i}}$.

\begin{proposition}[Dual-reward-bounded label pruning]
\label{prop:dual-reward-pruning}
For a label $\ForwardLabel{i}$ with $i\neq\DepotSink$, if $\KnapReducedCostLB{\ForwardLabel{i}}\ge0$, then $\CompletionCost{\ForwardLabel{i}}{\Completion}\ge0$ for every $\Completion\in\CompletionSet{\ForwardLabel{i}}$. Hence, $\ForwardLabel{i}$ can be safely discarded in standard pricing.
\end{proposition}

\begin{proof}
See Appendix~\ref{app:proof:prop-dual-reward-pruning}. 
\end{proof}

\paragraph{Farkas dual-reward pruning.}
For Farkas pricing, define $\FarkasDualRewardBound{\ForwardLabel{i}}$ by replacing each customer dual $\mu_n$ by the Farkas-ray component $\tilde\mu_n$ in the fractional-knapsack relaxation defining $\DualRewardBound{\ForwardLabel{i}}$. It is therefore solved by the same greedy ordering using $\PositivePart{\tilde\mu_n}/\DemandWeight{n}$. The corresponding completion lower bound is $\FarkasReducedCostLB{\ForwardLabel{i}}:=\LabelReducedCost{i}-\FarkasDualRewardBound{\ForwardLabel{i}}$. Unlike the standard-pricing bound, it contains no return-time term because time resources affect Farkas feasibility but do not appear in \eqref{eq:farkas-label-rc}.

\begin{proposition}[Farkas dual-reward label pruning]
\label{prop:farkas-dual-reward-pruning}
If $\FarkasReducedCostLB{\ForwardLabel{i}}\ge0$, then every $\Completion\in\CompletionSet{\ForwardLabel{i}}$ has nonnegative Farkas reduced cost. Hence, $\ForwardLabel{i}$ can be safely discarded in Farkas pricing.
\end{proposition}
\begin{proof}
See Appendix~\ref{app:proof:prop-farkas-dual-reward-pruning}. 
\end{proof}

\subsection{Parallel Pricing with Dynamic Route-Space Partitioning}
\label{sec:parallel-forward-pricing}
We implement the parallel-pricing scheme summarized in Section~\ref{sec:accel-summary} by partitioning $\BranchFeasibleRoutes{m}$ and dynamically reassigning open labels while preserving complete route-space coverage.

\paragraph{Initial route-space partition.}
At node $m$, define $\SourceNeighborSet{m}:=\{v_1^r:r\in\BranchFeasibleRoutes{m}\}$ as the possible first physical nodes after $\DepotSource$, and let $\PricingWorkerNum{m}$ denote the number of pricing processes. For each $v\in\SourceNeighborSet{m}$, let $\ForwardLabel{v}$ be the label obtained by extending the source label to $v$. Define $\mathcal U_m(\ForwardLabel{v})$ by \eqref{eq:reachable-customer-set} with the payload and service-time tests omitted. Thus, $\mathcal U_m(\ForwardLabel{v})$ contains the unserved customers that remain reachable on a relaxed continuation from $v$ to $\DepotSink$, either by direct truck service or by drone service from a reachable pad.

Using
$\DeadlineReachSet{\ForwardLabel{v}}
:=
\{n\in\mathcal U_m(\ForwardLabel{v}):
\LabelServLB{\ForwardLabel{v}}{n}\le\EffServUB{n}\}$, define the scheduling score
\begin{equation}\label{eq:balanced-work-score}
\begin{aligned}
\SourceWorkScore{m}{\ForwardLabel{v}}
:={}&
1
+|\mathcal U_m(\ForwardLabel{v})|
+\frac14
\left|
\left\{
j\in\Duplicate{\Nodes}:(v,j)\in\Duplicate{\Arcs}
\right\}
\right|\\
&+\frac12
\left|
\left\{
h\in\HubNodes:
\ShortestTruckTime{\PhysLoc{\ForwardLabel{v}}}{h}<+\infty,\;
\ShortestTruckTime{h}{\DepotSink}<+\infty,\;
\mathcal U_m(\ForwardLabel{v})\cap\DroneAccessNodes{h}\neq\emptyset
\right\}
\right|
+\frac12
|\DeadlineReachSet{\ForwardLabel{v}}|.
\end{aligned}
\end{equation}

The score estimates the amount of label-setting work associated with taking $v$ as the first node. Its four variable terms measure, respectively, the number of relaxed reachable customers, the number of outgoing transformed arcs, the number of reachable pads from which at least one customer in $\mathcal U_m(\ForwardLabel{v})$ can be served by drone, and the number of relaxed reachable customers satisfying their service-time limits. A larger $\SourceWorkScore{m}{v}$ therefore indicates a larger expected continuation search.
Initially, set $\SourceNeighborBlock{m}{k}=\emptyset$ for $k=1,\ldots,\PricingWorkerNum{m}$. The nodes in $\SourceNeighborSet{m}$ are considered in nonincreasing order of $\SourceWorkScore{m}{v}$. For each node $v$, assign it to the process $k$ having the smallest current total $\sum_{u\in\SourceNeighborBlock{m}{k}}\SourceWorkScore{m}{u}$; an empty block has total score $0$, and ties are broken by the process index.
The resulting assignment satisfies
\begin{equation}\label{eq:task-route-partition}
\SourceNeighborSet{m}
=
\biguplus_{k=1}^{\PricingWorkerNum{m}}
\SourceNeighborBlock{m}{k},
\qquad
\WorkerRouteSet{m}{k}
:=
\left\{
r\in\BranchFeasibleRoutes{m}:
v_1^r\in\SourceNeighborBlock{m}{k}
\right\},
\qquad
\BranchFeasibleRoutes{m}
=
\biguplus_{k=1}^{\PricingWorkerNum{m}}
\WorkerRouteSet{m}{k}.
\end{equation}

Here, $\WorkerRouteSet{m}{k}$ contains all branch-feasible routes whose first physical node belongs to $\SourceNeighborBlock{m}{k}$. Hence, every route in $\BranchFeasibleRoutes{m}$ is initially assigned to exactly one pricing process.

\paragraph{Dynamic route-space repartitioning.}
Each pricing process applies the forward label-setting algorithm to its assigned route set. If another process is idle, repartitioning is considered after every $1{,}000$ label extensions performed by the current process. For its task $\tau$, let $\OpenTaskLabels{\tau}$ denote the open nondominated labels and define $\TaskWorkEstimate{\tau}$ as the sum of $|\OpenTaskLabels{\tau}|$, the label extensions since the preceding checkpoint, and the full dominance comparisons over the same interval. Task $\tau$ is eligible for repartitioning if $|\OpenTaskLabels{\tau}|\ge500$, its processing time is at least $5$ seconds, and $\TaskWorkEstimate{\tau}\ge2{,}000$.
In standard pricing, define $\TaskClosureGap{\tau}:=\PositivePart{-\min_{\Lambda\in\OpenTaskLabels{\tau}}\KnapReducedCostLB{\Lambda}}$. In Farkas pricing, use $\TaskClosureGap{\tau}:=\PositivePart{-\min_{\Lambda\in\OpenTaskLabels{\tau}}\FarkasReducedCostLB{\Lambda}}$. In either case, repartitioning additionally requires $\TaskClosureGap{\tau}$ to exceed a small positive tolerance, avoiding repartitioning when the task is close to closure.

Among the eligible tasks, one donor $\tau$ is selected lexicographically by decreasing $\TaskWorkEstimate{\tau}$, decreasing $\TaskClosureGap{\tau}$, and decreasing $|\OpenTaskLabels{\tau}|$. Its open nondominated labels are then divided as evenly as possible among the donor and the currently idle pricing processes.
Each transferred label is removed from the donor and assigned with its complete state to exactly one receiving process. Since $\CompletionSet{\Lambda}$ is unchanged by a transfer, repartitioning changes only which process searches those completions. Each receiving process applies the applicable dominance rules to the transferred labels before continuing the search, and dominance comparisons remain task-local. All processes in the same pricing call use the same RMP dual solution, $\ActiveSRset{m}$, and inherited branching decisions.

Within each pricing call, standard pricing returns at most $\BatchLimit$ distinct negative-reduced-cost routes, whereas Farkas pricing returns at most $\FarkasBatchLimit$ routes violating the current certificate. If any are found, they are inserted into the RMP, which is reoptimized, and unfinished searches under the preceding dual solution or Farkas ray are discarded. Otherwise, with no pending transfer, pricing terminates only when every task is exhausted or all its open nondominated labels $\Lambda$ satisfy $\KnapReducedCostLB{\Lambda}\ge0$ in standard pricing or $\FarkasReducedCostLB{\Lambda}\ge0$ in Farkas pricing.

\subsection{Primal Heuristics and Column Pooling}
\label{sec:prim_heur}

\paragraph{Compact arc-based warm start and canonical decomposition.}
At the root node, the compact formulation in Section~\ref{arc-form} is solved under a fixed Gurobi optimization-time limit. If an incumbent solution is obtained, it is converted directly into route columns using the canonical representation in Section~\ref{sec:net-trans}: the selected truck arcs define each physical truck path, while the selected drone assignments at a visited pad $h$ define $\DroneBlock{h}{r}$, whose customers are ordered by $\prec_h$. The timing resources are recomputed from the resulting canonical paths, and each $\RouteCost{r}$ is evaluated according to \eqref{eq:route-cost-def}. The resulting columns are inserted into the initial RMP, and the corresponding set-partitioning solution initializes the root incumbent. If no usable route columns are obtained, the root RMP starts without route columns and proceeds directly to Farkas pricing as described in Section~\ref{sec:farkas-pricing}.

\paragraph{Branch-compatible route-pool incumbent search.}
The implementation maintains a global pool $\RoutePool$ of routes generated by pricing and the compact arc-based warm start. At node $m$, only routes satisfying the branching decisions at that node are retained, giving $\NodeRoutePool{m}:=(\RouteSubset{m}\cup\RoutePool)\cap\BranchFeasibleRoutes{m}$. The integer set-partitioning model defined in Section~\ref{sec:path-form} is then solved with $\RouteSet$ replaced by $\NodeRoutePool{m}$. Any feasible solution obtained from this restricted model is a valid incumbent candidate. Infeasibility of the restricted model only indicates that the currently available routes cannot form a feasible integer solution and does not imply infeasibility of node $m$.

\section{Test-Instance Generation and Computational Settings}
\label{app:test-instance-generation}

The depot is at $(0,0)$ in the region $[0,25]^2$ km. PS customer locations are sampled independently and uniformly over this region. PC instances use $G:=\max\{1,\lfloor|\CustNodes|/5\rfloor\}$ cluster centers sampled independently and uniformly over the region. Customers are assigned to centres cyclically in index order, with independent normal coordinate deviations of mean zero and standard deviation $0.7$ km; locations outside the region are resampled. Pads are the centroids of $K$-means clustering with $k=|\HubNodes|$.
Parcel weights (kg) follow $N(1,5^2)$ truncated to $[0.5,\DronePayload]$ for a uniformly selected subset of $\lfloor0.50|\CustNodes|\rfloor$ customers, and $N(10,5^2)$ truncated to $[6,+\infty)$ for the remainder. Among drone-eligible customers, $\lceil0.16|\CustNodes|\rceil$ are selected for the drone-only set $\MandatoryDroneCust$.
Service limits are $\ReqServUB{n}=\arriveLB{n}+\ServWindowOffset{n}$, with $\ServWindowOffset{n}\sim\operatorname{Unif}([30,90])$ minutes. Both vehicle speeds are $40$ km/h, with $\TruckPayload=50$ kg, $\DronePayload=2.3$ kg, $\DroneEndurance=30$ minutes, $(\TruckCost,\DroneCost)=(\$20,\$6)$, and objective weights $(\WeightDelay,\WeightReturn,\WeightCost)=(0.4,0.3,0.3)$.
Each BPC and MIQP run is limited to $1{,}800$ s. Standard and Farkas pricing batches contain at most $\BatchLimit=64$ and $\FarkasBatchLimit=16$ routes, respectively; each SR-separation round adds at most $\SRCutBatchLimit=32$ inequalities. The compact warm start is limited to $60$ s and each route-pool integer solve to $2$ s.

\section{Proofs}\label{app:proofs}
\subsection{Proof of Theorem~\ref{thm:cost-preserving-encoding}: Cost-Preserving Route Encoding}
\label{app:proof:thm-cost-preserving-encoding}

\begin{proof} Consider a feasible route $r\in\RouteSet$ with physical truck path $P^{\mathrm T}(r)=(v_0,\ldots,v_L)$, where $v_0=\DepotSource$ and $v_L=\DepotSink$. For each visited pad $h=v_\ell$, order $\DroneBlock{h}{r}$ as $(j_1,\ldots,j_q)$ under $\prec_h$. If $q\ge1$, replace $(h,v_{\ell+1})$ by $h,\Duplicate{j_1},\ldots,\Duplicate{j_q},v_{\ell+1}$; otherwise, retain the original truck arc. For a nonempty block, the entry, within-block, and exit arcs belong to $\CustHubArcsDuplicate$, $\CustDupCustDupArcs$, and $\CustDupCustArcs$, respectively, since $j_s\in\DroneAccessNodes{h}$, the sequence respects $\prec_h$, and $v_{\ell+1}\in\TruckAccessNodes{h}$. Applying this construction at every visited pad yields a source--sink path in $\TransGraph$.

The fixed truck path and canonical ordering uniquely determine the representation. The pad preceding each duplicated-node block identifies its launch location, even when a duplicated customer is accessible from several pads. Reassigning simultaneous sorties among identical drones leaves the representation unchanged.

Truck-served customers retain their physical nodes, while each drone-served customer is represented once by its duplicate. Thus, customer coverage, parcel demands, and the number of sorties launched at each pad are preserved. For $j\in\DroneBlock{h}{r}$, the delivery time remains $\TruckArrTime{h}{r}+\DroneTravelTime{h}{j}$, and the truck waits for $\max_{j\in\DroneBlock{h}{r}}\DroneTripTime{h}{j}$, with value zero for an empty block. These quantities coincide with the transformed-path resource updates. Since the physical truck movements and travel times are unchanged, all subsequent truck arrival times and the depot return time are also preserved.

The construction therefore preserves payload, drone-service, and timing feasibility. The customer service times, truck return time, and sortie count in \eqref{eq:route-cost-def} are unchanged, so the original route and its transformed path have the same cost. 
\end{proof}

\subsection{Proof of Proposition~\ref{improved_forward_dominance}: Same-Node Forward Dominance}
\label{app:proof:improved-forward-dominance}

\begin{proof}
Take any completion $\Completion\in\CompletionSet{\ForwardLabel{i}^{2}}$. Because the two labels end at the same transformed node and satisfy $\LabelHub{i}^{1}=\LabelHub{i}^{2}$, every transformed arc used by $\Completion$ from $\ForwardLabel{i}^{2}$ is also available from $\ForwardLabel{i}^{1}$. The inclusions $\RouteServedSet{\PartialPath{i}^{1}}\subseteq\RouteServedSet{\PartialPath{i}^{2}}$ and $\VisitedTruckNodes{i}^{1}\subseteq\VisitedTruckNodes{i}^{2}$ ensure that every customer and physical node used by $\Completion$ remains available from $\ForwardLabel{i}^{1}$. Since $\Completion$ adds the same customer demand and drone launches to both labels, the inequalities $\TruckWeight{i}^{1}\le\TruckWeight{i}^{2}$ and $\HubFlyDroneNum{i}^{1}\le\HubFlyDroneNum{i}^{2}$ preserve truck-payload and drone-fleet feasibility.

By the definition of $\BranchState{i}$, $\BranchState{i}^{1}=\BranchState{i}^{2}$ means that the two labels have the same current status for every $\BranchTogether{pq}$ and customer-conditioned transformed-arc requirement. Therefore, applying the same completion $\Completion$ preserves these branching requirements for both labels. For $\BranchSeparate{pq}$, if $\Completion$ caused $\ForwardLabel{i}^1$ to cover both $p$ and $q$, then $\RouteServedSet{\PartialPath{i}^{1}}\subseteq\RouteServedSet{\PartialPath{i}^{2}}$ would imply the same violation for $\ForwardLabel{i}^2$, contradicting the feasibility of $\Completion$.

The nonempty-route requirement is preserved by the disjunctive condition in the proposition. If $\RouteServedSet{\PartialPath{i}^{1}}\neq\emptyset$, the route completed from $\ForwardLabel{i}^1$ is already nonempty. Otherwise, $\RouteServedSet{\PartialPath{i}^{1}}=\RouteServedSet{\PartialPath{i}^{2}}=\emptyset$, and because $\Completion$ produces a nonempty route from $\ForwardLabel{i}^2$, it also covers at least one customer when applied to $\ForwardLabel{i}^1$.

Each time-resource update is nondecreasing in $\RealizedServTime{i}$, $\TruckArrHubTime{i}$, and $\TruckWaitHubTime{i}$ and uses the same travel or sortie times for both labels. Hence, the three time inequalities in the proposition imply that $\Completion$ produces no later service or return times from $\ForwardLabel{i}^{1}$. Every service envelope satisfied from $\ForwardLabel{i}^2$ is therefore also satisfied from $\ForwardLabel{i}^1$, so $\CompletionSet{\ForwardLabel{i}^{2}}\subseteq\CompletionSet{\ForwardLabel{i}^{1}}$.

It remains to compare reduced costs. For an active SR row $\pi$, let $t$ be the number of additional customers of $\pi$ covered by $\Completion$. Since $\RouteServedSet{\PartialPath{i}^{1}}\subseteq\RouteServedSet{\PartialPath{i}^{2}}$, we have $\NumCustPiCoverByLabel{i}{\pi}^{1}\le\NumCustPiCoverByLabel{i}{\pi}^{2}$, while elementarity gives $0\le t\le3-\NumCustPiCoverByLabel{i}{\pi}^{2}$. For label $a\in\{1,2\}$, the additional SR contribution is $-\nu_\pi\left(\left\lfloor(\NumCustPiCoverByLabel{i}{\pi}^{a}+t)/2\right\rfloor-\left\lfloor\NumCustPiCoverByLabel{i}{\pi}^{a}/2\right\rfloor\right)$. A direct check shows that $\ForwardLabel{i}^1$ can incur one additional future SR penalty only when $\NumCustPiCoverByLabel{i}{\pi}^{1}=1$, $\NumCustPiCoverByLabel{i}{\pi}^{2}=2$, and $t=1$, precisely for $\pi\in
\SRCorrectionSet{\ForwardLabel{i}^{1}}{\ForwardLabel{i}^{2}}$. Hence, the total future SR contribution from $\ForwardLabel{i}^1$ can exceed that from $\ForwardLabel{i}^2$ by at most $-\sum_{\pi\in\SRCorrectionSet{\ForwardLabel{i}^{1}}{\ForwardLabel{i}^{2}}}\nu_\pi$.

The completion $\Completion$ adds the same future customers and drone sorties to both labels, so its customer-dual and drone-cost increments are identical. Moreover, because feasible service times satisfy $\RealizedServTime{n}\ge\arriveLB{n}$, the quadratic delay and return-time costs are nondecreasing in their respective time resources. Therefore, all non-SR future cost increments from $\ForwardLabel{i}^1$ are no larger than those from $\ForwardLabel{i}^2$. The reduced-cost condition in the proposition offsets the largest possible additional SR penalty, yielding $\CompletionCost{\ForwardLabel{i}^{1}}{\Completion}\le\CompletionCost{\ForwardLabel{i}^{2}}{\Completion}$.

Hence, every feasible completion of $\ForwardLabel{i}^{2}$ is feasible from $\ForwardLabel{i}^{1}$ with no larger reduced cost, so $\ForwardLabel{i}^{2}$ can be safely discarded. 
\end{proof}

\subsection{Proof of Proposition~\ref{prop:physical-location-dominance}: Physical-Location Forward Dominance}
\label{app:proof:prop-physical-location-dominance}

\begin{proof}
Take any completion $\Completion\in\CompletionSet{\ForwardLabel{v}^{2}}$. Since $\ForwardLabel{v}^{2}$ does not end at $\DepotSink$, $\Completion$ contains at least one extension. Let $x$ be the first transformed node added by $\Completion$.
If $u=v$, both labels use the same first transformed arc.
Suppose instead that $u\neq v$. Since the two labels are location-comparable, their definition gives $\PhysLoc{\ForwardLabel{u}^{1}}=\PhysLoc{\ForwardLabel{v}^{2}}\in\HubNodes$ and $\BlockPos{\ForwardLabel{u}^{1}}\le\BlockPos{\ForwardLabel{v}^{2}}$. Let $h:=\PhysLoc{\ForwardLabel{u}^{1}}=\PhysLoc{\ForwardLabel{v}^{2}}$.
If $x\in\CustNodesDuplicate$, feasibility from $\ForwardLabel{v}^{2}$ implies that $\operatorname{rank}_{h}(\ServedCust{x})>\BlockPos{\ForwardLabel{v}^{2}}$.
Hence, $\BlockPos{\ForwardLabel{u}^{1}}\le\BlockPos{\ForwardLabel{v}^{2}}<\operatorname{rank}_{h}(\ServedCust{x})$, so the extension to $x$ also satisfies the canonical-order condition from $\ForwardLabel{u}^{1}$.
If $x\in\Nodes$, then the completion leaves the common pad $h$ and the truck next travels to $x$. Feasibility from $\ForwardLabel{v}^{2}$ therefore gives $x\in\TruckAccessNodes{h}$, i.e., $(h,x)\in\TruckArcs$. Since $\ForwardLabel{u}^{1}$ has the same physical truck location $h$, the same truck movement from $h$ to $x$ is also available from $\ForwardLabel{u}^{1}$.
Since $\BranchState{u}^{1}=\BranchState{v}^{2}$ and the two first extensions produce identical branch-state updates, the labels have the same branch state at $x$. 

From $x$ onward, use the same remaining transformed arcs as $\Completion$. Conditions~\eqref{eq:pld-cust}--\eqref{eq:pld-resource}, together with the feasibility arguments in the proof of Proposition~\ref{improved_forward_dominance}, preserve elementarity, capacity, the nonempty-route requirement, and branching feasibility. Condition~\eqref{eq:pld-time} further ensures that all subsequent service and departure times are no later from $\ForwardLabel{u}^{1}$, so $\Completion\in\CompletionSet{\ForwardLabel{u}^{1}}$.
For reduced cost, the corresponding completions add the same customers and drone sorties, so their customer-dual and drone-cost increments are identical, while the quadratic-delay increments from $\ForwardLabel{u}^{1}$ are no larger. The SR-correction argument in the proof of Proposition~\ref{improved_forward_dominance} bounds the possible additional SR contribution. Moreover, the final truck-return-time advantage is at least $\RetCredit{\ForwardLabel{u}^{1}}{\ForwardLabel{v}^{2}}$; at a common pad this follows from $\max\{\TruckWaitHubTime{u}^{1},T\}\le\max\{\TruckWaitHubTime{v}^{2},T\}$ for any additional maximum sortie duration $T$. Hence, condition~\eqref{eq:pld-cost} yields $\CompletionCost{\ForwardLabel{u}^{1}}{\Completion}\le\CompletionCost{\ForwardLabel{v}^{2}}{\Completion}$.

Thus, every feasible completion of $\ForwardLabel{v}^{2}$ has a feasible corresponding completion from $\ForwardLabel{u}^{1}$ with no larger reduced cost, so $\ForwardLabel{v}^{2}$ may be discarded. 
\end{proof}

\subsection{Proof of Proposition~\ref{prop:farkas-pricing-dominance}: Farkas-Pricing Dominance}
\label{app:proof:prop-farkas-pricing-dominance}

\begin{proof}
By the feasibility argument in the proof of Proposition~\ref{prop:physical-location-dominance}, conditions \eqref{eq:pld-cust}--\eqref{eq:pld-time} together with the first-extension condition in the proposition imply that every completion in $\CompletionSet{\ForwardLabel{v}^{2}}$ has a corresponding feasible completion from $\ForwardLabel{u}^{1}$.
For any such pair of corresponding completions, the same additional customers are covered in the same order. Since $\NumCustPiCoverByLabelSet{u}^{1}=\NumCustPiCoverByLabelSet{v}^{2}$, the active-SR count vectors remain identical after each corresponding extension. Thus, recursion \eqref{eq:farkas-label-rc} adds the same customer-ray and active-SR terms to both labels. Consequently, the difference between their completed Farkas reduced costs remains $\LabelReducedCost{u}^{1}-\LabelReducedCost{v}^{2}\le0$.

Therefore, every certificate-violating completion from $\ForwardLabel{v}^{2}$ has a feasible corresponding completion from $\ForwardLabel{u}^{1}$ with no larger Farkas reduced cost, so $\ForwardLabel{v}^{2}$ may be discarded. 
\end{proof}

\subsection{Proof of Proposition~\ref{prop:dual-reward-pruning}: Dual-Reward-Bounded Label Pruning}
\label{app:proof:prop-dual-reward-pruning}

\begin{proof}
Take any $\Completion\in\CompletionSet{\ForwardLabel{i}}$. The customers newly covered by $\Completion$ belong to $\ReachableCustSet{\ForwardLabel{i}}$ and have total demand at most $\ResidualPayload{\ForwardLabel{i}}$. Hence, setting $\KnapVar{n}=1$ for these customers and $\KnapVar{n}=0$ otherwise is feasible for the fractional-knapsack relaxation defining $\DualRewardBound{\ForwardLabel{i}}$, so their total positive customer-dual contribution is at most $\DualRewardBound{\ForwardLabel{i}}$.
The return-time contribution is at least $\ObjCoeffReturn\CompletionReturnLB{\ForwardLabel{i}}$, while all additional quadratic-delay, drone-cost, and active-SR contributions are nonnegative. Therefore, $\CompletionCost{\ForwardLabel{i}}{\Completion}\ge\LabelReducedCost{i}+\ObjCoeffReturn\CompletionReturnLB{\ForwardLabel{i}}-\DualRewardBound{\ForwardLabel{i}}=\KnapReducedCostLB{\ForwardLabel{i}}$. Hence, if $\KnapReducedCostLB{\ForwardLabel{i}}\ge0$, every feasible completion has nonnegative reduced cost and $\ForwardLabel{i}$ may be discarded. 
\end{proof}

\subsection{Proof of Proposition~\ref{prop:farkas-dual-reward-pruning}: Farkas Dual-Reward Label Pruning}
\label{app:proof:prop-farkas-dual-reward-pruning}

\begin{proof}
Take any $\Completion\in\CompletionSet{\ForwardLabel{i}}$. The customers newly covered by $\Completion$ belong to $\ReachableCustSet{\ForwardLabel{i}}$ and have total demand at most $\ResidualPayload{\ForwardLabel{i}}$. Hence, setting $\KnapVar{n}=1$ for these customers and $\KnapVar{n}=0$ otherwise is feasible for the Farkas fractional-knapsack relaxation. Their total positive customer-ray contribution is therefore at most $\FarkasDualRewardBound{\ForwardLabel{i}}$.
By \eqref{eq:farkas-label-rc}, the future customer-ray terms are bounded below by $-\FarkasDualRewardBound{\ForwardLabel{i}}$. Moreover, every future SR term is nonnegative because $\tilde\nu_\pi\le0$. Consequently, $\CompletionCost{\ForwardLabel{i}}{\Completion}\ge\LabelReducedCost{i}-\FarkasDualRewardBound{\ForwardLabel{i}}=\FarkasReducedCostLB{\ForwardLabel{i}}$. If $\FarkasReducedCostLB{\ForwardLabel{i}}\ge0$, no completion can have negative Farkas reduced cost, so $\ForwardLabel{i}$ may be discarded. 
\end{proof}

\end{document}